\documentclass[a4paper,12pt,reqno]{amsart}
\usepackage{amsfonts}
\usepackage{amsmath}
\usepackage{amssymb}
\usepackage[a4paper]{geometry}
\usepackage{mathrsfs}
\usepackage{csquotes}
\usepackage[colorlinks]{hyperref}
\renewcommand\eqref[1]{(\ref{#1})} 
\numberwithin{equation}{section}
\theoremstyle{plain}
\newtheorem{thm}{Theorem}[section]
\newtheorem{prop}[thm]{Proposition}
\newtheorem{cor}[thm]{Corollary}
\newtheorem{lem}[thm]{Lemma}
 \newtheorem{exa}[thm]{Example}
\theoremstyle{definition}

\newtheorem{rem}[thm]{Remark}

\renewcommand{\wp}{\mathfrak S}
\newcommand{\Rn}{\mathbb R^{n}}

\begin{document}

   \title[CKN, remainders and superweights for $L^{p}$-weighted Hardy inequalities]
   {Extended Caffarelli-Kohn-Nirenberg inequalities, and remainders, stability, and superweights for $L^{p}$-weighted Hardy inequalities}

\author[M. Ruzhansky]{Michael Ruzhansky}
\address{
  Michael Ruzhansky:
  \endgraf
  Department of Mathematics
  \endgraf
  Imperial College London
  \endgraf
  180 Queen's Gate, London SW7 2AZ
  \endgraf
  United Kingdom
  \endgraf
  {\it E-mail address} {\rm m.ruzhansky@imperial.ac.uk}
  }
\author[D. Suragan]{Durvudkhan Suragan}
\address{
  Durvudkhan Suragan:
  \endgraf
  Institute of Mathematics and Mathematical Modelling
  \endgraf
  125 Pushkin str.
  \endgraf
  050010 Almaty
  \endgraf
  Kazakhstan
  \endgraf
  {\it E-mail address} {\rm suragan@math.kz}
  }
\author[N. Yessirkegenov]{Nurgissa Yessirkegenov}
\address{
  Nurgissa Yessirkegenov:
  \endgraf
  Institute of Mathematics and Mathematical Modelling
  \endgraf
  125 Pushkin str.
  \endgraf
  050010 Almaty
  \endgraf
  Kazakhstan
  \endgraf
  and
  \endgraf
  Department of Mathematics
  \endgraf
  Imperial College London
  \endgraf
  180 Queen's Gate, London SW7 2AZ
  \endgraf
  United Kingdom
  \endgraf
  {\it E-mail address} {\rm n.yessirkegenov15@imperial.ac.uk}
  }

\thanks{The authors were supported in parts by the EPSRC
 grant EP/K039407/1 and by the Leverhulme Grant RPG-2014-02,
 as well as by the MESRK grant 5127/GF4. No new data was collected or
generated during the course of research.}


     \begin{abstract}
     In this paper we give an extension of the classical Caffarelli-Kohn-Nirenberg inequalities: we show that
      for $1<p,q<\infty$, $0<r<\infty$ with $p+q\geq r$, $\delta\in[0,1]\cap\left[\frac{r-q}{r},\frac{p}{r}\right]$ with $\frac{\delta r}{p}+\frac{(1-\delta)r}{q}=1$ and $a$, $b$, $c\in\mathbb{R}$ with $c=\delta(a-1)+b(1-\delta)$, and for all functions $f\in C_{0}^{\infty}(\Rn\backslash\{0\})$ we have
  $$
\||x|^{c}f\|_{L^{r}(\Rn)}
\leq \left|\frac{p}{n-p(1-a)}\right|^{\delta} \left\||x|^{a}\nabla f\right\|^{\delta}_{L^{p}(\Rn)}
\left\||x|^{b}f\right\|^{1-\delta}_{L^{q}(\Rn)}
$$
for $n\neq p(1-a)$, where the constant $\left|\frac{p}{n-p(1-a)}\right|^{\delta}$ is sharp for $p=q$ with $a-b=1$ or $p\neq q$ with $p(1-a)+bq\neq0$.
   In the critical case $n=p(1-a)$ we have
 $$
\left\||x|^{c}f\right\|_{L^{r}(\Rn)}
\leq p^{\delta} \left\||x|^{a}\log|x|\nabla f\right\|^{\delta}_{L^{p}(\Rn)}
\left\||x|^{b}f\right\|^{1-\delta}_{L^{q}(\Rn)}.
$$
Moreover, we also obtain anisotropic versions of these inequalities which can be conveniently formulated in the language of Folland and Stein's homogeneous groups.
Consequently, we obtain remainder estimates for $L^{p}$-weighted Hardy inequalities on homogeneous groups, which are also new in the Euclidean setting of $\Rn$. The critical Hardy inequalities of logarithmic type and uncertainty type principles on homogeneous groups are obtained. Moreover, we investigate another improved version of $L^{p}$-weighted Hardy inequalities involving a distance and stability estimates. The relation between the critical and the subcritical Hardy inequalities on homogeneous groups is also investigated. We also establish sharp Hardy type inequalities
in $L^{p}$, $1<p<\infty$, with superweights, i.e. with
the weights of the form $\frac{(a+b|x|^{\alpha})^{\frac{\beta}{p}}}{|x|^{m}}$ allowing for different choices of $\alpha$ and $\beta$.
There are two reasons why we call the appearing weights the superweights: the arbitrariness of the choice of any homogeneous quasi-norm and a wide range of parameters.

     \end{abstract}
     \maketitle

     \tableofcontents

\section{Introduction}
\label{SEC:intro}

The aim of this paper is to give an extension of the classical Caffarelli-Kohn-Nirenberg (CKN) inequalities \cite{CKN84} with respect to ranges of parameters and to investigate the remainders and stability of the weighted $L^p$-Hardy inequalities. Moreover, our methods also provide sharp constants for the CKN inequality for known ranges of parameters as well as give an improvement by replacing the full gradient by the radial derivative. We also obtain the critical case of the CKN inequality with logarithmic terms, and investigate the remainders and other properties in the case when CKN inequalities reduce to the weighted Hardy inequalities.
For the latter, we also establish $L^p$ weighted Hardy inequalities with more general weights of the form $\frac{(a+b|x|^{\alpha})^{\frac{\beta}{p}}}{|x|^{m}}$, allowing for different choices of $m$, $\alpha$ and $\beta$.

\subsection{Extended Caffarelli-Kohn-Nirenberg inequalities}

Let us recall the classical Caffarelli-Kohn-Nirenberg inequality \cite{CKN84}:

\begin{thm}\label{clas_CKN}
Let $n\in\mathbb{N}$ and let $p$, $q$, $r$, $a$, $b$, $d$, $\delta\in \mathbb{R}$ such that $p,q\geq1$, $r>0$, $0\leq\delta\leq1$, and
\begin{equation}\label{clas_CKN0}
\frac{1}{p}+\frac{a}{n},\, \frac{1}{q}+\frac{b}{n},\, \frac{1}{r}+\frac{c}{n}>0
\end{equation}
where $c=\delta d + (1-\delta) b$. Then there exists a positive constant $C$ such that
\begin{equation}\label{clas_CKN1}
\||x|^{c}f\|_{L^{r}(\Rn)}\leq C \||x|^{a}|\nabla f|\|^{\delta}_{L^{p}(\Rn)} \||x|^{b}f\|^{1-\delta}_{L^{q}(\Rn)}
\end{equation}
holds for all $f\in C_{0}^{\infty}(\Rn)$, if and only if the following conditions hold:
\begin{equation}\label{clas_CKN2}
\frac{1}{r}+\frac{c}{n}=\delta \left(\frac{1}{p}+\frac{a-1}{n}\right)+(1-\delta)\left(\frac{1}{q}+\frac{b}{n}\right),
\end{equation}
\begin{equation}\label{clas_CKN3}
a-d\geq 0 \quad {\rm if} \quad \delta>0,
\end{equation}
\begin{equation}\label{clas_CKN4}
a-d\leq 1 \quad {\rm if} \quad \delta>0 \quad {\rm and} \quad \frac{1}{r}+\frac{c}{n}=\frac{1}{p}+\frac{a-1}{n}.
\end{equation}
\end{thm}

The first aim of this paper is to extend the CKN-inequalities for general functions with respect to widening the range of indices \eqref{clas_CKN0}. Moreover, another improvement will be achieved by replacing the full gradient $\nabla f$ in \eqref{clas_CKN1} by the radial derivative
$\mathcal{R}f=\frac{\partial f}{\partial r}$. It turns out that such improved versions can be establsihed with sharp constants, and to hold both in the isotropic and anisotropic settings.

To compare with Theorem \ref{clas_CKN} let us first formulate the isotropic version of our extension in the usual setting of $\Rn$.

\begin{thm}\label{clas_CKN-2}
Let $n\in\mathbb{N}$, $1<p,q<\infty$, $0<r<\infty$, with $p+q\geq r$, $\delta\in[0,1]\cap\left[\frac{r-q}{r},\frac{p}{r}\right]$ and $a$, $b$, $c\in\mathbb{R}$. Assume that $\frac{\delta r}{p}+\frac{(1-\delta)r}{q}=1$,  $c=\delta(a-1)+b(1-\delta)$. If $n\neq p(1-a)$, then for any function $f\in C_{0}^{\infty}(\mathbb{R}^{n}\backslash\{0\})$ we have
\begin{equation}\label{EQ:aq}
\||x|^{c}f\|_{L^{r}(\Rn)}
\leq \left|\frac{p}{n-p(1-a)}\right|^{\delta} \left\||x|^{a}\left(\frac{x}{|x|}\cdot\nabla f\right)\right\|^{\delta}_{L^{p}(\Rn)}
\left\||x|^{b}f\right\|^{1-\delta}_{L^{q}(\Rn)}.
\end{equation}

In the critical case $n=p(1-a)$ for any function $f\in C_{0}^{\infty}(\mathbb{R}^{n}\backslash\{0\})$ we have
\begin{equation}\label{CKN_rem3a}
\left\||x|^{c}f\right\|_{L^{r}(\Rn)}
\leq p^{\delta} \left\||x|^{a}\log|x|\left(\frac{x}{|x|}\cdot\nabla f\right)\right\|^{\delta}_{L^{p}(\Rn)}
\left\||x|^{b}f\right\|^{1-\delta}_{L^{q}(\Rn)},
\end{equation}
for any homogeneous quasi-norm $|\cdot|$. If $|\cdot|$ is the Euclidean norm on $\Rn$, inequalities
\eqref{EQ:aq} and \eqref{CKN_rem3a} imply, respectively,
\begin{equation}\label{EQ:aqe}
\||x|^{c}f\|_{L^{r}(\Rn)}
\leq \left|\frac{p}{n-p(1-a)}\right|^{\delta} \left\||x|^{a}\nabla f\right\|^{\delta}_{L^{p}(\Rn)}
\left\||x|^{b}f\right\|^{1-\delta}_{L^{q}(\Rn)}
\end{equation}
for $n\neq p(1-a)$, and
\begin{equation}\label{CKN_rem3ae}
\left\||x|^{c}f\right\|_{L^{r}(\Rn)}
\leq p^{\delta} \left\||x|^{a}\log|x|\nabla f\right\|^{\delta}_{L^{p}(\Rn)}
\left\||x|^{b}f\right\|^{1-\delta}_{L^{q}(\Rn)},
\end{equation}
for $n=p(1-a)$. The inequality \eqref{EQ:aq} holds for any homogeneous quasi-norm $|\cdot|$, and the constant
$\left|\frac{p}{n-p(1-a)}\right|^{\delta}$ is sharp for $p=q$ with $a-b=1$ or for $p\neq q$ with $p(1-a)+bq\neq0$. Furthermore, the constants $\left|\frac{p}{n-p(1-a)}\right|^{\delta}$ and $p^{\delta}$ are sharp for $\delta=\overline{0,1}$.
\end{thm}

Note that if the conditions \eqref{clas_CKN0}
hold, then the inequality \eqref{EQ:aqe} is contained in the family of Caffarelli-Kohn-Nirenberg inequalities in Theorem \ref{clas_CKN}.
However, already in this case, if we require $p=q$ with $a-b=1$ or $p\neq q$ with $p(1-a)+bq\neq0$, then \eqref{EQ:aqe} yields the inequality \eqref{clas_CKN1} with sharp constant. Moreover, the constants $\left|\frac{p}{n-p(1-a)}\right|^{\delta}$ and $p^{\delta}$ are sharp for $\delta=0$ or $\delta=1$. Our conditions $\frac{\delta r}{p}+\frac{(1-\delta)r}{q}=1$ and $c=\delta(a-1)+b(1-\delta)$ imply the condition \eqref{clas_CKN2} of the Theorem \ref{clas_CKN}, as well as \eqref{clas_CKN3}-\eqref{clas_CKN4} which are all necessary for having estimates of this type, at least under the conditions \eqref{clas_CKN0}.

If the conditions \eqref{clas_CKN0} are not satisfied, then the inequality \eqref{EQ:aqe} is not covered by Theorem \ref{clas_CKN}. So, this gives an extension of Theorem \ref{clas_CKN} with respect to the range of parameters. Let us give an example:

\begin{exa}\label{CKN_example} Let us take $1<p=q=r<\infty$, $a=-\frac{n-2p}{p}$, $b=-\frac{n}{p}$ and $c=-\frac{n-\delta p}{p}$.
Then by \eqref{EQ:aqe}, for all $f\in C_{0}^{\infty}(\mathbb{R}^{n}\backslash\{0\})$ we have  the inequality
\begin{equation}\label{CKN_example1}
\left\|\frac{f}{|x|^{\frac{n-\delta p}{p}}}\right\|_{L^{p}(\Rn)}\leq
\left\|\frac{\nabla f}{|x|^{\frac{n-2p}{p}}}\right\|^{\delta}_{L^{p}(\Rn)}
\left\|\frac{f}{|x|^{\frac{n}{p}}}\right\|^{1-\delta}_{L^{p}(\Rn)}, \quad 1<p<\infty, \;0\leq\delta\leq1,
\end{equation}
where $\nabla$ is the standard gradient in $\Rn$. Since we have
$$\frac{1}{q}+\frac{b}{n}=\frac{1}{p}+\frac{1}{n}\left(-\frac{n}{p}\right)=0,$$
we see that \eqref{clas_CKN0} fails, so that the inequality \eqref{CKN_example1} is not covered by Theorem \ref{clas_CKN}. Moreover, in this case, $p=q$ with $a-b=1$ hold true, so that the constant $\left|\frac{p}{n-p(1-a)}\right|^{\delta}=1$ in the inequality \eqref{CKN_example1} is sharp.
\end{exa}

Although these results are new already in the usual setting of $\Rn$, our techniques apply well also for the anisotropic structures. Consequently, it is convenient  to work in the setting of homogeneous groups developed by Folland and Stein \cite{FS-Hardy} with an idea of  emphasising general results of harmonic analysis depending only of the group and dilation structures.
In particular, in this way we obtain results on the anisotropic $\Rn$, on the Heisenberg group, general stratified groups, graded groups, etc. In the special case of stratified groups (or homogeneous Carnot groups), other formulations using horizontal gradient are possible, and we refer to \cite{Ruzhansky-Suragan:Layers} and especially to \cite{Ruzhansky-Suragan:JDE} for versions of such results and the discussion of the corresponding literature.

The improved versions of the Caffarelli-Kohn-Nirenberg inequality for radially symmetric functions with respect to the range of parameters was investigated in \cite{NDD12}. In \cite{ZhHD15} and \cite{HZh11}, weighted Hardy type inequalities were obtained for the generalised Baouendi-Grushin vector fields, which is when $\gamma=0$ gives the standard gradient in $\Rn$. We also refer to \cite{HNZh11}, \cite{Han15} for weighted Hardy inequalities on the Heisenberg group, to \cite{HZhD11} and \cite{ZhHD14} on the H-type groups, and a recent paper \cite{Yacoub17} on Lie groups of polynomial growth as well as to references therein.

In Section \ref{SEC:prelim} we very briefly recall the necessary notions and fix the notation in more detail. Assuming the notation there, Theorem \ref{clas_CKN-2} is the special case of the following theorem that we prove in this paper:

\begin{thm}\label{THM:CKN-i}
Let $\mathbb{G}$ be a homogeneous group
of homogeneous dimension $Q$. Let $|\cdot|$ be an arbitrary homogeneous quasi-norm on $\mathbb{G}$. Let $1<p,q<\infty$, $0<r<\infty$ with $p+q\geq r$, $\delta\in[0,1]\cap\left[\frac{r-q}{r},\frac{p}{r}\right]$ and $a$, $b$, $c\in\mathbb{R}$. Assume that $\frac{\delta r}{p}+\frac{(1-\delta)r}{q}=1$ and $c=\delta(a-1)+b(1-\delta)$. Then
for all $f\in C_{0}^{\infty}(\mathbb{G}\backslash\{0\})$ we have the following Caffarelli-Kohn-Nirenberg type inequalities, with $\mathcal{R}:=\frac{d}{d|x|}$ being the radial derivative:
If $Q\neq p(1-a)$, then
$$
\||x|^{c}f\|_{L^{r}(\mathbb{G})}
\leq \left|\frac{p}{Q-p(1-a)}\right|^{\delta} \left\||x|^{a}\mathcal{R}f\right\|^{\delta}_{L^{p}(\mathbb{G})}
\left\||x|^{b}f\right\|^{1-\delta}_{L^{q}(\mathbb{G})},
$$
where the constant $\left|\frac{p}{Q-p(1-a)}\right|^{\delta}$ is sharp for $p=q$ with $a-b=1$ or $p\neq q$ with $p(1-a)+bq\neq0$.
If $Q=p(1-a)$, then
$$
\left\||x|^{c}f\right\|_{L^{r}(\mathbb{G})}
\leq p^{\delta} \left\||x|^{a}\log|x|\mathcal{R}f\right\|^{\delta}_{L^{p}(\mathbb{G})}
\left\||x|^{b}f\right\|^{1-\delta}_{L^{q}(\mathbb{G})}.
$$ Moreover, the constants $\left|\frac{p}{Q-p(1-a)}\right|^{\delta}$ and $p^{\delta}$ are sharp for $\delta=\overline{0,1}$.
\end{thm}

\subsection{$L^{p}$-weighted Hardy inequalities}

Let us recall the following $L^{p}$-weighted Hardy inequality
\begin{equation}\label{Lp_Hardy}
\int_{\Rn}\frac{|\nabla f(x)|^{p}}{|x|^{\alpha p}}dx\geq\left(\frac{n-p-\alpha p}{p}\right)^{p}\int_{\Rn}\frac{|f(x)|^{p}}{|x|^{(\alpha+1)p}}dx
\end{equation}
for every function $f\in C_{0}^{\infty}(\Rn)$, where $-\infty<\alpha<\frac{n-p}{p}$ and $2\leq p<n$. The inequality \eqref{Lp_Hardy} is a special case of the Caffarelli-Kohn-Nirenberg inequalities \cite{CKN84}, recalled also in Theorem \ref{clas_CKN}. Since in this paper we are also interested in remainder estimates for the $L^{p}$-weighted Hardy inequality, let us introduce known results in this direction. Overall, the study of remainders in Hardy and other related inequalities is a classical topic going back to \cite{Brez1, Brez2, BV97}.

Ghoussoub and Moradifam \cite{GM08} proved that there exists no strictly positive function $V\in C^{1}(0,\infty)$ such that the inequality
$$\int_{\Rn}|\nabla f|^{2}dx\geq\left(\frac{n-2}{2}\right)^{2}\int_{\Rn}\frac{|f|^{2}}{|x|^{2}}dx+\int_{\Rn}V(|x|)|f|^{2}dx$$
holds for any $f\in W^{1,2}(\Rn)$. Cianchi and Ferone \cite {CF08} showed that for all $1<p<n$ there exists a constant $C=C(p,n)$ such that
$$\int_{\Rn}|\nabla f|^{p}dx\geq\left(\frac{n-p}{p}\right)^{p}\int_{\Rn}\frac{|f|^{p}}{|x|^{p}}dx\,(1+Cd_{p}(f)^{2p^{*}})$$
holds for all real-valued weakly differentiable functions $f$ in $\Rn$ such that $f$ and $|\nabla f|\in L^{p}(\Rn)$ go to zero at infinity. Here
$$d_{p}f=\underset{c\in \mathbb{R}}{\rm inf}\frac{\|f-c|x|^{-\frac{n-p}{p}}\|_{L^{p^{*},\infty}(\Rn)}}{\|f\|_{L^{p^{*},p}(\Rn)}}$$ with $p^{*}=\frac{np}{n-p}$, and $L^{\tau, \sigma}(\Rn)$ is the Lorentz space for $0<\tau\leq \infty$ and $1\leq\sigma\leq\infty$.
In the case of a bounded domain $\Omega$, Wang and Willem \cite{WW03} for $p=2$ and Abdellaoui, Colorado and Peral \cite{ACP05} for $1<p<\infty$ investigated the improved type of \eqref{Lp_Hardy} (see also \cite{ST15a} and \cite{ST15b} for more details).

For more general Lie group discussions of above inequalities we refer to recent papers \cite{Ruzhansky-Suragan:Layers}, \cite{Ruzhansky-Suragan:squares} and \cite{Ruzhansky-Suragan:JDE} as well as references therein.

Sometimes the improved versions of different inequalities, or remainder estimates, are called stability of the inequality if the estimates depend on certain distances: see, e.g. \cite{BJOS16} for stability of trace theorems, \cite{CFW13} for stability of Sobolev inequalities, etc.

We also note that Sano and Takahashi obtained the improved version of \eqref{Lp_Hardy} in \cite{ST15a} for $\Omega=\Rn$ and $\alpha=0$ and then in \cite{ST15b} for any $-\infty<\alpha<\frac{n-p}{p}$:
Let $n\geq 3$, $2\leq p<n$ and $-\infty<\alpha<\frac{n-p}{p}$. Let $N\in \mathbb{N}$, $t\in (0,1)$, $\gamma<\min\{1-t, \frac{p-N}{p}\}$ and  $\delta=N-n+\frac{N}{1-t-\gamma}\left(\gamma+\frac{n-p-\alpha p}{p}\right)$.
Then there exists a constant $C>0$ such that the inequality
$$\int_{\Rn}\frac{|\nabla f|^{p}}{|x|^{\alpha p}}dx-\left(\frac{n-p-\alpha p}{p}\right)^{p}\int_{\Rn}\frac{|f|^{p}}{|x|^{p(\alpha+1)}}dx
\geq C\frac{\left(\int_{\Rn}|f|^{\frac{N}{1-t-\gamma}}|x|^{\delta}dx\right)^{\frac{p(1-t-\gamma)}{Nt}}}
{\left(\int_{\Rn}|f|^{p}|x|^{-\alpha p}dx\right)^{\frac{1-t}{t}}}
$$
holds for any radially symmetric function $f\in W_{0,\alpha}^{1,p}(\mathbb{R}^{n})$, $f\neq 0$.

For the convenience of the reader we now briefly recapture the main results of this part of the  paper, formulating them directly in the anisotropic cases following the notation recalled in Section \ref{SEC:prelim}.
Thus, we show that for a homogeneous group $\mathbb{G}$ of homogeneous dimension $Q$ and any homogeneous quasi-norm $|\cdot|$ we have the following results:
\begin{itemize}
\item ({\bf Remainder estimates for the $L^{p}$-weighted Hardy inequality})
Let $2\leq p<Q$, $-\infty<\alpha<\frac{Q-p}{p}$ and $\delta_{1}=Q-p-\alpha p-\frac{Q+pb}{p}$, $\delta_{2}=Q-p-\alpha p-\frac{bp}{p-1}$ for any $b\in\mathbb{R}$. Then
for all functions $f\in C_{0}^{\infty}(\mathbb{G}\backslash\{0\})$ we have
$$\int_{\mathbb{G}}\frac{|\mathcal{R}f(x)|^{p}}{|x|^{\alpha p}}dx-\left(\frac{Q-p-\alpha p}{p}\right)^{p}\int_{\mathbb{G}}\frac{|f(x)|^{p}}{|x|^{(\alpha+1)p}}dx$$
$$
\geq C_{p} \frac{\left(\int_{\mathbb{G}}|f(x)|^{p}|x|^{\delta_{1}}dx\right)^{p}}
{\left(\int_{\mathbb{G}}|f(x)|^{p}|x|^{\delta_{2}}dx\right)^{p-1}},
$$
where $C_{p}=c_{p}
\left|\frac{Q(p-1)-pb}{p^{2}}\right|^{p}$, $\mathcal{R}:=\frac{d}{d|x|}$ is the radial derivative and $c_{p}=\underset{0<t\leq1/2}{\rm min}((1-t)^{p}-t^{p}+pt^{p-1})$. This family is a new result already in the standard setting of $\Rn$.
\item ({\bf Stability of Hardy inequalities})
Let $2\leq p<Q$ and $-\infty<\alpha<\frac{Q-p}{p}$. Then for all radial functions $f\in C_{0}^{\infty}(\mathbb{G}\backslash\{0\})$ we have the stability estimate
$$\int_{\mathbb{G}}\frac{|\mathcal{R}f(x)|^{p}}{|x|^{\alpha p}}dx-\left(\frac{Q-p-\alpha p}{p}\right)^{p}\int_{\mathbb{G}}\frac{|f(x)|^{p}}{|x|^{(\alpha+1)p}}dx$$
$$
\geq c_{p}\left(\frac{p-1}{p}\right)^{p}\sup_{R>0}d_{R}(f,c_{f}(R)f_{\alpha})^{p},
$$
where $c_{f}(R)=R^{\frac{Q-p-\alpha p}{p}}\widetilde{f}(R)$ with $f(x)=\widetilde{f}(r)$, $r=|x|$,
$\mathcal{R}:=\frac{d}{d|x|}$ is the radial derivative, $c_{p}$ is defined in Lemma \ref{FrS}, $f_{\alpha}$ and $d_{R}(\cdot,\cdot)$ are defined in \eqref{aremterm7} and \eqref{aremterm8}, respectively.

\item ({\bf Critical Hardy inequalities of logarithmic type})
Let $1<\gamma<\infty$ and let $\max\{1,\gamma-1\}<p<\infty$.
Then for all $f\in C_{0}^{\infty}(\mathbb{G}\backslash\{0\})$ and all $R>0$ we have
$$\frac{p}{\gamma-1}\left\|
\frac{\mathcal{R}f}{|x|^{\frac{Q-p}{p}}\left(\log\frac{R}{|x|}\right)^{\frac{\gamma-p}{p}}} \right\|_{L^{p}(\mathbb{G})}
\geq
\left\|\frac{f-f_{R}}{|x|^{\frac{Q}{p}}\left(\log\frac{R}{|x|}\right)^{\frac{\gamma}{p}}}\right\|_{L^{p}(\mathbb{G})},
$$
where $f_{R}=f\left(R\frac{x}{|x|}\right)$,  where $\mathcal{R}:=\frac{d}{d|x|}$ is the radial derivative, and the constant $\frac{p}{\gamma-1}$ is optimal. In the abelian case, this result was obtained in \cite{MOW15}. In the case $\gamma=p$ this result on the homogeneous group was proved in \cite{Ruzhansky-Suragan:critical}.
\item ({\bf Uncertainty inequalities})
Let $1<p<\infty$ and $q>1$ be such that
$\frac{1}{p}+\frac{1}{q}=\frac{1}{2}$. Let $1<\gamma<\infty$ and $\max\{1,\gamma-1\}<p<\infty$. Then for any $R>0$ and $f\in C_{0}^{\infty}(\mathbb{G}\backslash\{0\})$ we have the uncertainty inequalities
$$
\left\|
\frac{\mathcal{R}f}{|x|^{\frac{Q-p}{p}}\left(\log\frac{R}{|x|}\right)^{\frac{\gamma-p}{p}}} \right\|_{L^{p}(\mathbb{G})}\|f\|_{L^{q}(\mathbb{G})}
\geq\frac{\gamma-1}{p}\left\|\frac{f(f-f_{R})}
{|x|^{\frac{Q}{p}}\left(\log\frac{R}{|x|}\right)^{\frac{\gamma}{p}}} \right\|_{L^{2}(\mathbb{G})},
$$
where $\mathcal{R}:=\frac{d}{d|x|}$ is the radial derivative (see \eqref{EQ:Euler}). Moreover,
$$
\left\|
\frac{\mathcal{R}f}{|x|^{\frac{Q-p}{p}}\left(\log\frac{R}{|x|}\right)^{\frac{\gamma-p}{p}}} \right\|_{L^{p}(\mathbb{G})}\left\|\frac{f-f_{R}}
{|x|^{\frac{Q}{p'}}\left(\log\frac{R}{|x|}\right)^{2-\frac{\gamma}{p}}}
\right\|_{L^{p'}(\mathbb{G})}$$$$\geq\frac{\gamma-1}{p}\left\|\frac{f-f_{R}}
{|x|^{\frac{Q}{2}}\log\frac{R}{|x|}}\right\|^{2}_{L^{2}(\mathbb{G})}
$$
holds for $\frac{1}{p}+\frac{1}{p'}=1$.

\item ({\bf Relation between critical and subcritical Hardy inequalities})
Let $Q\geq m+1$, $m\geq2$. Let $|\cdot|$ be a homogeneous quasi-norm. Then for any nonnegative radial function $g\in C_{0}^{1}(B^{m}(0,R)\backslash\{0\})$, there exists a nonnegative radial function $f\in C_{0}^{1}(B^{Q}(0,1)\backslash\{0\})$ such that
$$\int_{B^{Q}(0,1)}|\mathcal{R}f|^{m}dx-\left(\frac{Q-m}{m}\right)^{m}\int_{B^{Q}(0,1)}\frac{|f|^{m}}{|x|^{m}}dx$$
$$
=\frac{|\sigma|}
{|\widetilde{\sigma}|}\left(\frac{Q-m}{m-1}\right)^{m-1}$$
$$\times\left(\int_{B^{m}(0,R)}|\mathcal{R}g|^{m}dz-
\left(\frac{m-1}{m}\right)^{m}\int_{B^{m}(0,R)}\frac{|g|^{m}}{|z|^{m}\left(\log\frac{Re}{|z|}\right)^{m}}dz\right)
$$
holds true, where $\mathcal{R}:=\frac{d}{d|x|}$ is the radial derivative, $|\sigma|$ and $|\widetilde{\sigma}|$ are $Q-1$ and $m-1$ dimensional surface measure of the unit sphere, respectively.
\end{itemize}

\subsection{$L^p$-Hardy inequalities with superweights}

The classical Hardy inequalities and their extensions, such as the Caffarelli-Kohn-Nirenberg inequalities, usually involve the weights of the form $\frac{1}{|x|^{m}}$. In this paper, we also consider the weights of the form $\frac{(a+b|x|^{\alpha})^{\frac{\beta}{p}}}{|x|^{m}}$ allowing for different choices of $\alpha$ and $\beta$. If $\alpha=0$ or $\beta=0$, this reduces to traditional weights. So, we are interested in the case when $\alpha\beta\not=0$ and, in fact, we obtain two families of inequalities depending on whether $\alpha\beta>0$ or $\alpha\beta<0$. Moreover, $|\cdot|$ in these expressions can be an arbitrary homogeneous quasi-norm and the constants for the obtained inequalities are sharp. The freedom in choosing parameters $\alpha,\beta, a,b,m$ and a quasi-norm led us to calling these weights the `superweights' in this context.

Again, the obtained estimates will include both the isotropic and anisotropic settings of $\mathbb R^{n}$, for which our range of obtained estimates appears also to be new. Namely, already in the Euclidean case of $\mathbb R^{n}$ with the Euclidean norm, they extend the inequalities that have been known for $p=2$ for some range of parameters from \cite{GM11} to the full range of $1<p<\infty$.

Therefore, we can again work on the homogeneous groups.
To summarise, on a homogeneous group $\mathbb{G}$
with homogeneous dimension $Q$ for any homogeneous quasi-norm $|\cdot|$ on $\mathbb{G}$, all $a,b>0$ and $1<p<\infty$ we prove that
\begin{itemize}
\item If $\alpha \beta>0$ and $pm\leq Q-p$, then for all $f\in C_{0}^{\infty}(\mathbb{G}\backslash\{0\})$, we have
\begin{equation}\label{intro1} \frac{Q-pm-p}{p}
\left\|\frac{(a+b|x|^{\alpha})^{\frac{\beta}{p}}}{|x|^{m+1}}f\right\|_{L^{p}(\mathbb{G})}
\leq\left\|\frac{(a+b|x|^{\alpha})^{\frac{\beta}{p}}}{|x|^{m}}\mathcal{R}f\right\|_{L^{p}(\mathbb{G})} .
\end{equation}
If $Q\neq pm+p$, then the constant $\frac{Q-pm-p}{p}$ is sharp.

\item If $\alpha \beta<0$ and $pm-\alpha\beta\leq Q-p$, then for all $f\in C_{0}^{\infty}(\mathbb{G}\backslash\{0\})$, we have
\begin{equation} \label{intro2}
\frac{Q-pm+\alpha\beta-p}{p}
\left\|\frac{(a+b|x|^{\alpha})^{\frac{\beta}{p}}}{|x|^{m+1}}f\right\|_{L^{p}(\mathbb{G})}
\leq\left\|\frac{(a+b|x|^{\alpha})^{\frac{\beta}{p}}}{|x|^{m}}\mathcal{R}f\right\|_{L^{p}(\mathbb{G})}
.
\end{equation}
If $Q\neq pm+p-\alpha\beta$, then the constant $\frac{Q-pm+\alpha\beta-p}{p}$ is sharp.
\end{itemize}

As noted before, the weights in the inequalities \eqref{intro1} and \eqref{intro2} are called superweights since the constants $\frac{Q-pm-p}{p}$ in \eqref{intro1} and $\frac{Q-pm+\alpha\beta-p}{p}$ in \eqref{intro2} are sharp for arbitrary homogeneous quasi-norm $|\cdot|$ of $\mathbb{G}$ and wide range of choices of the allowed parameters  $\alpha, \beta, a, b$ and $m$. Directly from the inequalities \eqref{intro1} and \eqref{intro2}, choosing different $\alpha, \beta, a, b, m$ and $Q$, one can obtain a number of Hardy type inequalities which have various consequences
and applications. For instance, in the Abelian (isotropic or anisotropic) case ${\mathbb G}=(\mathbb R^{n},+)$, we have
$Q=n$, so for any quasi-norm $|\cdot|$ on $\mathbb R^{n}$, all $a,b>0$ and $1<p<\infty$ these imply
new inequalities. Thus, if $\alpha \beta>0$ and $pm\leq Q-p$, then for all $f\in C_{0}^{\infty}(\mathbb R^{n}\backslash\{0\})$, we have
\begin{equation}\label{intro3} \frac{n-pm-p}{p}
\left\|\frac{(a+b|x|^{\alpha})^{\frac{\beta}{p}}}{|x|^{m+1}}f\right\|_{L^{p}(\mathbb R^{n})}
\leq\left\|\frac{(a+b|x|^{\alpha})^{\frac{\beta}{p}}}{|x|^{m}}\frac{df}{d|x|}\right\|_{L^{p}(\mathbb R^{n})}
\end{equation}
with the constant being sharp for $n\neq pm+p$.

If $\alpha \beta<0$ and $pm-\alpha\beta\leq n-p$, then for all $f\in C_{0}^{\infty}(\mathbb{G}\backslash\{0\})$, we have
\begin{equation} \label{intro4}
\frac{n-pm+\alpha\beta-p}{p}
\left\|\frac{(a+b|x|^{\alpha})^{\frac{\beta}{p}}}{|x|^{m+1}}f\right\|_{L^{p}(\mathbb R^{n})}
\leq\left\|\frac{(a+b|x|^{\alpha})^{\frac{\beta}{p}}}{|x|^{m}}\frac{df}{d|x|} \right\|_{L^{p}(\mathbb R^{n})}
\end{equation}
with the sharp constant for $n\neq pm+p-\alpha\beta$.
In the case of the standard Euclidean distance $|x|=\sqrt{x^{2}_{1}+\ldots+x^{2}_{n}}$ by using the Schwartz inequality from the inequalities \eqref{intro3} and \eqref{intro4} we obtain that
if $\alpha \beta>0$ and $pm\leq Q-p$, then for all $f\in C_{0}^{\infty}(\mathbb R^{n}\backslash\{0\})$
\begin{equation}\label{intro5} \frac{n-pm-p}{p}
\left\|\frac{(a+b|x|^{\alpha})^{\frac{\beta}{p}}}{|x|^{m+1}}f\right\|_{L^{p}(\mathbb R^{n})}
\leq\left\|\frac{(a+b|x|^{\alpha})^{\frac{\beta}{p}}}{|x|^{m}}\nabla f\right\|_{L^{p}(\mathbb R^{n})}
\end{equation}
with the constant sharp for $n\neq pm+p$.

If $\alpha \beta<0$ and $pm-\alpha\beta\leq n-p$, then for all $f\in C_{0}^{\infty}(\mathbb{G}\backslash\{0\})$, we have
\begin{equation} \label{intro6}
\frac{n-pm+\alpha\beta-p}{p}
\left\|\frac{(a+b|x|^{\alpha})^{\frac{\beta}{p}}}{|x|^{m+1}}f\right\|_{L^{p}(\mathbb R^{n})}
\leq\left\|\frac{(a+b|x|^{\alpha})^{\frac{\beta}{p}}}{|x|^{m}}\nabla f\right\|_{L^{p}(\mathbb R^{n})}
\end{equation}
with the sharp constant for $n\neq pm+p-\alpha\beta$. The $L^{2}$-version, that is, when $p=2$ the inequalities \eqref{intro5} and \eqref{intro6} were obtained in \cite{GM11}. We also shall note that these inequalities have interesting applications in theory of ODE (see \cite[Theorem 2.1]{GM11}).

In Section \ref{SEC:2} we give the main result of this part and give its short proof. Some higher order versions of the obtained inequalities are discussed briefly in Section \ref{Sec3}.

In Section \ref{SEC:prelim} we briefly recall the main concepts of homogeneous groups and fix the notation. In Section \ref{SEC:critHardy} we present critical Hardy inequalities of logarithmic type and uncertainty type principles on homogeneous groups. The remainder estimates for $L^{p}$-weighted Hardy inequalities on homogeneous groups are proved in Section \ref{SEC:rem_estimates}. Moreover, in Section \ref{stab} we also investigate another improved version of $L^{p}$-weighted Hardy inequalities involving a distance. In Section \ref{SEC:crit_subcrit_con} the relation between the critical and the subcritical Hardy inequalities on homogeneous groups is investigated. In Section \ref{SEC:CKN} we introduce Caffarelli-Kohn-Nirenberg type inequalities on homogenous groups and prove their extended version.

\section{Preliminaries}
\label{SEC:prelim}

In this section we very briefly recall the necessary notation concerning the setting of homogeneous
groups following Folland and Stein \cite{FS-Hardy} as well as a recent treatise \cite{FR}.
We also recall a few other facts that will be used in the proofs.
A connected simply connected Lie group $\mathbb G$ is called a {\em homogeneous group} if
its Lie algebra $\mathfrak{g}$ is equipped with a family of the following dilations:
$$D_{\lambda}={\rm Exp}(A \,{\rm ln}\lambda)=\sum_{k=0}^{\infty}
\frac{1}{k!}({\rm ln}(\lambda) A)^{k},$$
where $A$ is a diagonalisable positive linear operator on $\mathfrak{g}$,
and every $D_{\lambda}$ is a morphism of $\mathfrak{g}$,
that is,
$$\forall X,Y\in \mathfrak{g},\, \lambda>0,\;
[D_{\lambda}X, D_{\lambda}Y]=D_{\lambda}[X,Y],$$
holds. We recall that $Q := {\rm Tr}\,A$ is called the homogeneous dimension of $\mathbb G$.

A homogeneous group is a nilpotent (Lie) group and exponential mapping $\exp_{\mathbb G}:\mathfrak g\to\mathbb G$ of this group is a global diffeomorphism.
Thus, this implies the dilation structure, and this dilation is denoted by $D_{\lambda}x$ or just by $\lambda x$, on homogeneous groups.

Then we have
\begin{equation}
|D_{\lambda}(S)|=\lambda^{Q}|S| \quad {\rm and}\quad \int_{\mathbb{G}}f(\lambda x)
dx=\lambda^{-Q}\int_{\mathbb{G}}f(x)dx.
\end{equation}
Here $dx$ is the Haar measure on homogeneous groups $\mathbb{G}$ and $|S|$ is the volume of a measurable set $S\subset \mathbb{G}$. The Haar measure on a homogeneous group $\mathbb{G}$ is the standard Lebesgue measure for $\Rn$ (see, for example \cite[Proposition 1.6.6]{FR}).

Let $|\cdot|$ be a homogeneous quasi-norm on $\mathbb G$.
Then the quasi-ball centred at $x\in\mathbb{G}$ with radius $R > 0$ is defined by
$$B(x,R):=\{y\in \mathbb{G}: |x^{-1}y|<R\}.$$
The following notation will be also used in this paper
$$B^{c}(x,R):=\{y\in \mathbb{G}: |x^{-1}y|\geq R\}.$$
We refer to \cite{FS-Hardy} for the proof of the following important polar decomposition on homogeneous Lie groups, which can be also found in \cite[Section 3.1.7]{FR}:
there is a (unique)
positive Borel measure $\sigma$ on the
unit quasi-sphere
\begin{equation}\label{EQ:sphere}
\wp:=\{x\in \mathbb{G}:\,|x|=1\},
\end{equation}
so that for every $f\in L^{1}(\mathbb{G})$ we have
\begin{equation}\label{EQ:polar}
\int_{\mathbb{G}}f(x)dx=\int_{0}^{\infty}
\int_{\wp}f(ry)r^{Q-1}d\sigma(y)dr.
\end{equation}
Let us now fix a basis $\{X_{1},\ldots,X_{n}\}$ of a Lie algebra $\mathfrak{g}$
such that
$$AX_{k}=\nu_{k}X_{k}$$
for every $k$, so that the matrix $A$ can be taken to be
$A={\rm diag} (\nu_{1},\ldots,\nu_{n}).$
Then every $X_{k}$ is homogeneous of degree $\nu_{k}$ and
$$
Q=\nu_{1}+\cdots+\nu_{n}.
$$
The decomposition of ${\exp}_{\mathbb{G}}^{-1}(x)$ in $\mathfrak g$ defines the vector
$$e(x)=(e_{1}(x),\ldots,e_{n}(x))$$
by the formula
$${\exp}_{\mathbb{G}}^{-1}(x)=e(x)\cdot \nabla\equiv\sum_{j=1}^{n}e_{j}(x)X_{j},$$
where $\nabla=(X_{1},\ldots,X_{n})$.
It implies the equality
$$x={\exp}_{\mathbb{G}}\left(e_{1}(x)X_{1}+\ldots+e_{n}(x)X_{n}\right).$$
Taking into account the homogeneity and denoting $x=ry,\,y\in \wp,$ one has
$$
e(x)=e(ry)=(r^{\nu_{1}}e_{1}(y),\ldots,r^{\nu_{n}}e_{n}(y)).
$$
So we have
\begin{equation*}
\frac{d}{d|x|}f(x)=\frac{d}{dr}f(ry)=
 \frac{d}{dr}f({\exp}_{\mathbb{G}}
\left(r^{\nu_{1}}e_{1}(y)X_{1}+\ldots
+r^{\nu_{n}}e_{n}(y)X_{n}\right)).
\end{equation*}
We use the notation
\begin{equation}\label{EQ:Euler}
\mathcal{R} :=\frac{d}{dr},
\end{equation}
that is,
\begin{equation}\label{dfdr}
	\frac{d}{d|x|}f(x)=\mathcal{R}f(x), \;\forall x\in \mathbb G,
\end{equation}
for any homogeneous quasi-norm $|x|$ on $\mathbb G$.
Let us recall the following lemma, which will be used in our proof.
\begin{lem}[\cite{FrS08}]
\label{FrS}
Let $p\geq2$ and let $a$, $b$ be real numbers. Then there exists $c_{p}>0$ such that
$$|a-b|^{p}\geq|a|^{p}-p|a|^{p-2}ab+c_{p}|b|^{p}$$
holds, where $c_{p}=\underset{0<t\leq1/2}{\rm min}((1-t)^{p}-t^{p}+pt^{p-1})$ is sharp in this inequality.
\end{lem}
We will also use the following result (see \cite{ORS16} and \cite{Ruzhansky-Suragan:L2-CKN}) with anisotropic Caffarelli-Kohn-Nirenberg inequality:

\begin{thm}[\cite{ORS16}]
\label{CKN}
Let $\mathbb{G}$ be a homogeneous group
of homogeneous dimension $Q$. Let $|\cdot|$ be a homogeneous quasi-norm. Let $a,b\in\mathbb{R}$, and $f\in C_{0}^{\infty}(\mathbb{G}\backslash\{0\})$. Then we have
\begin{equation}\label{CKN1}
\left|\frac{Q-(a+b+1)}{p}\right|\int_{\mathbb{G}}\frac{|f|^{p}}{|x|^{a+b+1}}dx\leq\left(\int_{\mathbb{G}}
\frac{|\mathcal{R}f|^{p}}{|x|^{ap}}dx\right)^{\frac{1}{p}}\left(\int_{\mathbb{G}}
\frac{|f|^{p}}{|x|^{\frac{bp}{p-1}}}dx\right)^{\frac{p-1}{p}},
\end{equation}
where $\mathcal{R}$ is defined in \eqref{EQ:Euler}, $1<p<\infty$, and the constant $\left|\frac{Q-(a+b+1)}{p}\right|$ is sharp.
\end{thm}

\section{On remainder estimates of anisotropic $L^{p}$-weighted Hardy inequalities}
\label{SEC:rem_estimates}

In this section we obtain a family of remainder estimates in the weighted $L^p$-Hardy inequalities, with a freedom of choosing the parameter $b\in\mathbb R$. The obtained remainder estimates are new already in the standard setting of $\mathbb R^n$.

\begin{thm}\label{aremterm}
Let $\mathbb{G}$ be a homogeneous group
of homogeneous dimension $Q\geq3$. Let $|\cdot|$ be a homogeneous quasi-norm. Let $2\leq p<Q$, $-\infty<\alpha<\frac{Q-p}{p}$ and $\delta_{1}=Q-p-\alpha p-\frac{Q+pb}{p}$, $\delta_{2}=Q-p-\alpha p-\frac{bp}{p-1}$ for any $b\in\mathbb{R}$. Then
for all functions $f\in C_{0}^{\infty}(\mathbb{G}\backslash\{0\})$ we have
$$\int_{\mathbb{G}}\frac{|\mathcal{R}f(x)|^{p}}{|x|^{\alpha p}}dx-\left(\frac{Q-p-\alpha p}{p}\right)^{p}\int_{\mathbb{G}}\frac{|f(x)|^{p}}{|x|^{(\alpha+1)p}}dx$$
\begin{equation}\label{aremterm1}
\geq C_{p} \frac{\left(\int_{\mathbb{G}}|f(x)|^{p}|x|^{\delta_{1}}dx\right)^{p}}
{\left(\int_{\mathbb{G}}|f(x)|^{p}|x|^{\delta_{2}}dx\right)^{p-1}},
\end{equation}
where $\mathcal{R}$ is defined in \eqref{EQ:Euler}, $C_{p}=c_{p}
\left|\frac{Q(p-1)-pb}{p^{2}}\right|^{p}$ and $c_{p}=\underset{0<t\leq1/2}{\rm min}((1-t)^{p}-t^{p}+pt^{p-1})$.
\end{thm}
\begin{rem}\label{aremterm_rem1} Since the inequality \eqref{aremterm1} holds for any $b\in\mathbb{R}$, choosing $b=\frac{Q(p-1)}{p}$ so that $C_{p}=0$, we obtain the $L^{p}$-weighted Hardy inequalities on homogeneous groups:
\begin{multline}\label{aremterm13}
\int_{\mathbb{G}}\frac{|\mathcal{R}f(x)|^{p}}{|x|^{\alpha p}}dx\geq\left(\frac{Q-p-\alpha p}{p}\right)^{p}\int_{\mathbb{G}}\frac{|f(x)|^{p}}{|x|^{(\alpha+1)p}}dx, \\
-\infty<\alpha<\frac{Q-p}{p},\;\; 2\leq p<Q,
\end{multline}
for all functions $f\in C_{0}^{\infty}(\mathbb{G}\backslash\{0\})$.
In the abelian case $\mathbb{G}=(\Rn,+)$ with $Q=n$, the inequality \eqref{aremterm13} gives the $L^{p}$-weighted Hardy inequalities for any quasi-norm on $\Rn$: For any function $f\in C_{0}^{\infty}(\Rn\backslash\{0\})$ we have
$$\int_{\Rn}\left|\frac{x}{|x|}\cdot\nabla f(x)\right|^{p}|x|^{-\alpha p}dx\geq\left(\frac{n-p-\alpha p}{p}\right)^{p}\int_{\Rn}\frac{|f(x)|^{p}}{|x|^{p(\alpha+1)}}dx,$$
where $-\infty<\alpha<\frac{n-p}{p}$ and $2\leq p<n$. By Schwarz's inequality with the standard Euclidean distance $|x|=\sqrt{x_{1}^{2}+x_{2}^{2}+...+x_{n}^{2}}$, we obtain the Euclidean form of the $L^{p}$-weighted Hardy inequalities on $\Rn$:
\begin{multline*}
\int_{\Rn}\frac{|\nabla f(x)|^{p}}{|x|^{\alpha p}}dx\geq\left(\frac{n-p-\alpha p}{p}\right)^{p}\int_{\Rn}\frac{|f(x)|^{p}}{|x|^{(\alpha+1)p}}dx, \\ -\infty<\alpha<\frac{n-p}{p},\;\; 2\leq p<n,
\end{multline*}
for any function $f\in C_{0}^{\infty}(\Rn\backslash\{0\})$, where $\nabla$ is the standard gradient in $\Rn$.
\end{rem}
\begin{rem}\label{aremterm_rem2}
We also note that in the abelian case, \eqref{aremterm1} implies a new remainder estimate for any quasi-norm on $\Rn$: For any function $f\in C_{0}^{\infty}(\Rn\backslash\{0\})$ and for any $b\in\mathbb{R}$, we obtain
$$\int_{\Rn}\left|\frac{x}{|x|}\cdot\nabla f(x)\right|^{p}|x|^{-\alpha p}dx-\left(\frac{n-p-\alpha p}{p}\right)^{p}\int_{\Rn}\frac{|f(x)|^{p}}{|x|^{p(\alpha+1)}}dx$$
\begin{equation}\label{aremterm11}
\geq C_{p} \frac{\left(\int_{\Rn}|f(x)|^{p}|x|^{\delta_{1}}dx\right)^{p}}
{\left(\int_{\Rn}|f(x)|^{p}|x|^{\delta_{2}}dx\right)^{p-1}},\quad 2\leq p<n,\;-\infty<\alpha<\frac{n-p}{p}.
\end{equation}
As in Remark \ref{aremterm_rem1}, by Schwarz's inequality with the standard Euclidean distance, we obtain the Euclidean version of the remainder estimate for $L^{p}$-weighted Hardy inequalities:
$$\int_{\Rn}\frac{\left|\nabla f(x)\right|^{p}}{|x|^{\alpha p}}dx-\left(\frac{n-p-\alpha p}{p}\right)^{p}\int_{\Rn}\frac{|f(x)|^{p}}{|x|^{(\alpha+1)p}}dx$$
\begin{equation}\label{aremterm12}
\geq C_{p} \frac{\left(\int_{\Rn}|f(x)|^{p}|x|^{\delta_{1}}dx\right)^{p}}
{\left(\int_{\Rn}|f(x)|^{p}|x|^{\delta_{2}}dx\right)^{p-1}},\quad 2\leq p<n,\;-\infty<\alpha<\frac{n-p}{p},
\end{equation}
for every function $f\in C_{0}^{\infty}(\Rn\backslash\{0\})$ and for any $b\in\mathbb{R}$, where $\nabla$ is the standard gradient in $\Rn$.

Thus, we note that the remainder estimate \eqref{aremterm12} is new already in the standard setting of $\Rn$.
\end{rem}
\begin{proof}[Proof of Theorem \ref{aremterm}] First let us show the statement of Theorem \ref{aremterm} for a radial function $f\in C_{0}^{\infty}(\mathbb{G}\backslash\{0\})$. Since $f$ is radial, $f$ can be represented as $f(x)=\widetilde{f}(|x|)$. By Brezis-V\'{a}zquez's idea (\cite{BV97}), we define
\begin{equation}\label{aremterm2}
\widetilde{g}(r)=r^{\frac{Q-p-\alpha p}{p}}\widetilde{f}(r).
\end{equation}
Since $\widetilde{f}=\widetilde{f}(r)\in C_{0}^{\infty}(0,\infty)$ and $\alpha<\frac{Q-p}{p}$, we obtain $\widetilde{g}(0)=0$ and $\widetilde{g}(+\infty)=0$. We set $g(x)=\widetilde{g}(|x|)$ for $x\in \mathbb{G}$.
Introducing polar coordinates $(r,y)=(|x|, \frac{x}{\mid x\mid})\in (0,\infty)\times\wp$ on $\mathbb{G}$ and using \eqref{EQ:polar}, we have
$$J:=\int_{\mathbb{G}}|\mathcal{R}f|^{p}|x|^{-\alpha p}dx-\left(\frac{Q-p-\alpha p}{p}\right)^{p}\int_{\mathbb{G}}\frac{|f|^{p}}{|x|^{p(\alpha+1)}}dx$$
$$=|\sigma|\int_{0}^{\infty}\left|\frac{d}{dr}\widetilde{f}(r)\right|^{p}r^{-\alpha p+Q-1}dr-|\sigma|
\left(\frac{Q-p-\alpha p}{p}\right)^{p}\int_{0}^{\infty}|\widetilde{f}(r)|^{p}r^{-p(\alpha+1)+Q-1}dr$$
$$=|\sigma|\int_{0}^{\infty}\left|\left(\frac{Q-p-\alpha p}{p}\right)r^{-\frac{Q-\alpha p}{p}}\widetilde{g}(r)
-r^{-\frac{Q-p-\alpha p}{p}}\frac{d}{dr}\widetilde{g}(r)\right|^{p}r^{Q-1-\alpha p}dr$$
$$-|\sigma|\left(\frac{Q-p-\alpha p}{p}\right)^{p}\int_{0}^{\infty}|\widetilde{g}(r)|^{p}r^{-1}dr,$$
where $|\sigma|$ is the $Q-1$ dimensional surface measure of the unit quasi-sphere.
Here applying Lemma \ref{FrS} to the integrand of the first term in the last expression above, we get
$$\left|\left(\frac{Q-p-\alpha p}{p}\right)r^{-\frac{Q-\alpha p}{p}}\widetilde{g}(r)-r^{-\frac{Q-p-\alpha p}{p}}\frac{d}{dr}\widetilde{g}(r)\right|^{p}r^{Q-1-\alpha p}$$
$$\geq\left(\left(\frac{Q-p-\alpha p}{p}\right)^{p}r^{-Q+\alpha p}|\widetilde{g}(r)|^{p}\right)r^{Q-1-\alpha p}$$
$$-p\left(\frac{Q-p-\alpha p}{p}\right)^{p-1}|\widetilde{g}(r)|^{p-2}\widetilde{g}(r)\frac{d}{dr}\widetilde{g}(r)
r^{-(\frac{Q-\alpha p}{p})(p-1)}r^{-(\frac{Q-p-\alpha p}{p})}r^{Q-1-\alpha p}$$
$$+c_{p}\left|\frac{d}{dr}\widetilde{g}(r)\right|^{p}r^{-Q+p+\alpha p}r^{Q-1-\alpha p}$$
$$=\left(\frac{Q-p-\alpha p}{p}\right)^{p}r^{-1}|\widetilde{g}(r)|^{p}-p\left(\frac{Q-p-\alpha p}{p}\right)^{p-1}|\widetilde{g}(r)|^{p-2}\widetilde{g}(r)\frac{d}{dr}\widetilde{g}(r)$$
$$+c_{p}\left|\frac{d}{dr}\widetilde{g}(r)\right|^{p}r^{p-1}.$$
Since $\widetilde{g}(0)=\widetilde{g}(+\infty)=0$ and $p\geq2$, we note that
$$p\int_{0}^{\infty}|\widetilde{g}(r)|^{p-2}\widetilde{g}(r)\frac{d}{dr}\widetilde{g}(r)dr=\int_{0}^{\infty}
\frac{d}{dr}(|\widetilde{g}(r)|^{p})dr=0.$$
This gives a \enquote{ground state representation} (\cite{FrS08}) of the Hardy difference $J$:
\begin{equation}\label{aremterm3}
J\geq c_{p} |\sigma|\int_{0}^{\infty}\left|\frac{d}{dr}\widetilde{g}(r)\right|^{p}r^{p-1}dr=
c_{p}\int_{\mathbb{G}}|\mathcal{R}g(x)|^{p}|x|^{p-Q}dx.
\end{equation}
Putting $a=\frac{Q-p}{p}$ in \eqref{CKN1}, we obtain for any $b\in\mathbb{R}$, that
\begin{multline}\label{CKN2}
\left|\frac{Q(p-1)-pb}{p^{2}}\right|\int_{\mathbb{G}}|g|^{p}
|x|^{-\frac{Q+pb}{p}}dx \\
\leq\left(\int_{\mathbb{G}}
|\mathcal{R}g|^{p}|x|^{p-Q}dx\right)^{\frac{1}{p}}\left(\int_{\mathbb{G}}
|g|^{p}|x|^{-\frac{bp}{p-1}}dx\right)^{\frac{p-1}{p}}.
\end{multline}
It gives that
\begin{equation}\label{aremterm4}J\geq c_{p}\int_{\mathbb{G}}|\mathcal{R}g(x)|^{p}|x|^{p-Q}dx\geq c_{p}
\left|\frac{Q(p-1)-pb}{p^{2}}\right|^{p}\frac{\left(\int_{\mathbb{G}}|g|^{p}|x|^{-\frac{Q+pb}{p}}
dx\right)^{p}}{\left(\int_{\mathbb{G}}
|g|^{p}|x|^{-\frac{bp}{p-1}}dx\right)^{p-1}}.\end{equation}
Taking into account that $g(x)=\widetilde{g}(|x|)$, $x\in\mathbb{G}$, and \eqref{aremterm2}, one calculates
$$\int_{\mathbb{G}}|x|^{-\frac{Q+pb}{p}}|g(x)|^{p}dx=|\sigma|\int_{0}^{\infty}r^{Q-p-\alpha p}|\widetilde{f}(r)|^{p}r^{-\frac{Q+pb}{p}}r^{Q-1}dr
$$
$$=\int_{\mathbb{G}}|f(x)|^{p}|x|^{Q-p-\alpha p-\frac{Q+pb}{p}}dx=
\int_{\mathbb{G}}|f(x)|^{p}|x|^{\delta_{1}}dx.$$
On the other hand,
$$\int_{\mathbb{G}}|x|^{-\frac{bp}{p-1}}|g(x)|^{p}dx=|\sigma|\int_{0}^{\infty}r^{Q-p-\alpha p}|\widetilde{f}(r)|^{p}r^{-\frac{bp}{p-1}}r^{Q-1}dr$$
$$=\int_{\mathbb{G}}|f(x)|^{p}|x|^{Q-p-\alpha p-\frac{bp}{p-1}}dx=\int_{\mathbb{G}}|f(x)|^{p}|x|^{\delta_{2}}dx.$$
Putting these into \eqref{aremterm4}, we obtain
\begin{equation}\label{aremterm4_01}J\geq c_{p}
\left|\frac{Q(p-1)-pb}{p^{2}}\right|^{p}
\frac{\left(\int_{\mathbb{G}}|f(x)|^{p}|x|^{\delta_{1}}dx\right)^{p}}
{\left(\int_{\mathbb{G}}|f(x)|^{p}|x|^{\delta_{2}}dx\right)^{p-1}}.
\end{equation}
Now let us prove it for non-radial functions. We consider the radial function for a non-radial function $f$:
\begin{equation}\label{aremterm4_1}U(r)=\left(\frac{1}{|\sigma|}\int_{\wp}|f(ry)|^{p}d\sigma(y)\right)^{\frac{1}{p}}.\end{equation}
Using the H\"{o}lder inequality, we calculate
$$\frac{d}{dr}U(r)=\frac{1}{p}\left(\frac{1}{|\sigma|}\int_{\wp}|f(ry)|^{p}d\sigma(y)\right)^{\frac{1}{p}-1}
\frac{1}{|\sigma|}\int_{\wp}p|f(ry)|^{p-2}f(ry)\overline{\frac{d}{dr}f(ry)}d\sigma(y)$$
$$\leq\left(\frac{1}{|\sigma|}\int_{\wp}|f(ry)|^{p}d\sigma(y)\right)^{\frac{1}{p}-1}
\frac{1}{|\sigma|}\int_{\wp}|f(ry)|^{p-1}\left|\frac{d}{dr}f(ry)\right|d\sigma(y)$$
$$\leq \left(\frac{1}{|\sigma|}\int_{\wp}|f(ry)|^{p}d\sigma(y)\right)^{\frac{1}{p}-1}
\frac{1}{|\sigma|}\left(\int_{\wp}\left|\frac{d}{dr}f(ry)\right|^{p}d\sigma(y)\right)^{\frac{1}{p}}
\left(\int_{\wp}|f(ry)|^{p}d\sigma(y)\right)^{\frac{p-1}{p}}$$
$$=\left(\frac{1}{|\sigma|}\int_{\wp}\left|\frac{d}{dr}f(ry)\right|^{p}d\sigma(y)\right)^{\frac{1}{p}}.$$
Thus, we have
$$
\frac{d}{dr}U(r)\leq\left(\frac{1}{|\sigma|}\int_{\wp}\left|\frac{d}{dr}f(ry)\right|^{p}d\sigma(y)\right)^{\frac{1}{p}}.
$$
It follows that
$$|\sigma|\int_{0}^{\infty}\left|\frac{d}{dr}U(r)\right|^{p}r^{Q-1-\alpha p}dr\leq
|\sigma|\int_{0}^{\infty}\frac{1}{|\sigma|}\int_{\wp}\left|\frac{d}{dr}f(ry)\right|^{p}r^{Q-1-\alpha p}d\sigma(y)dr$$
$$=\int_{\mathbb{G}}\left|\mathcal{R}f\right|^{p}|x|^{-\alpha p}dx,$$
that is,
\begin{equation}\label{aremterm5}
\int_{\mathbb{G}}\left|\mathcal{R}U\right|^{p}|x|^{-\alpha p}dx\leq\int_{\mathbb{G}}\left|\mathcal{R}f\right|^{p}|x|^{-\alpha p}dx.
\end{equation}
In view of \eqref{aremterm4_1}, we obtain
$$\int_{\mathbb{G}}|U(|x|)|^{p}|x|^{\theta}dx=|\sigma|\int_{0}^{\infty}|U(r)|^{p}r^{\theta+Q-1}dr$$
\begin{equation}\label{aremterm6}=|\sigma|\int_{0}^{\infty}\frac{1}{|\sigma|}\int_{\wp}|f(ry)|^{p}d\sigma(y)r^{\theta+Q-1}dr=
\int_{\mathbb{G}}|f(x)|^{p}|x|^{\theta}dx \end{equation}
for any $\theta\in\mathbb{R}$.
Then, it is easy to see that \eqref{aremterm5} and \eqref{aremterm6} imply that \eqref{aremterm1} holds also for all non-radial functions.
\end{proof}

\section{Stability of anisotropic $L^{p}$-weighted Hardy inequalities}
\label{stab}
In this section we establish a remainder estimate in the $L^{p}$-weighted Hardy inequality involving the distance to the set of extremisers: estimates of such type are known as stability estimates in the literature.
Let us denote
\begin{equation}\label{aremterm7}
f_{\alpha}(x)=|x|^{-\frac{Q-p-\alpha p}{p}}
\end{equation}
for $-\infty<\alpha<\frac{Q-p}{p}$, and we set
\begin{equation}\label{aremterm8}
d_{R}(f,g):=\left(\int_{\mathbb{G}}\frac{|f(x)-g(x)|^{p}}{\left|\log\frac{R}{|x|}\right|^{p}|x|^{(\alpha+1)p}}dx\right)^{\frac{1}{p}}
\end{equation}
for functions $f$ and $g$ for which the integral in \eqref{aremterm8} is finite.
\begin{thm}\label{aremterm_thm}
Let $\mathbb{G}$ be a homogeneous group
of homogeneous dimension $Q\geq3$. Let $|\cdot|$ be a homogeneous quasi-norm. Let $2\leq p<Q$ and $-\infty<\alpha<\frac{Q-p}{p}$. Then for all radial functions $f\in C_{0}^{\infty}(\mathbb{G}\backslash\{0\})$ we have
$$\int_{\mathbb{G}}\frac{|\mathcal{R}f|^{p}}{|x|^{\alpha p}}dx-\left(\frac{Q-p-\alpha p}{p}\right)^{p}\int_{\mathbb{G}}\frac{|f|^{p}}{|x|^{(\alpha+1)p}}dx$$
\begin{equation}\label{aremterm9}
\geq c_{p}\left(\frac{p-1}{p}\right)^{p}\sup_{R>0}d_{R}(f,c_{f}(R)f_{\alpha})^{p},
\end{equation}
where $c_{f}(R)=R^{\frac{Q-p-\alpha p}{p}}\widetilde{f}(R)$ with $f(x)=\widetilde{f}(r)$, $|x|=r$, $\mathcal{R}:=\frac{d}{d|x|}$ is the radial derivative, $c_{p}$ is defined in Lemma \ref{FrS}, $f_{\alpha}$ and $d_{R}(\cdot,\cdot)$ are defined in \eqref{aremterm7} and \eqref{aremterm8}, respectively.
\end{thm}
\begin{proof}[Proof of Theorem \ref{aremterm_thm}] Since $p\geq2$, as in  \eqref{aremterm3} in the proof of Theorem \ref{aremterm}, we have
$$J(f)=\int_{\mathbb{G}}|\mathcal{R}f|^{p}|x|^{-\alpha p}dx-\left(\frac{Q-p-\alpha p}{p}\right)^{p}\int_{\mathbb{G}}\frac{|f|^{p}}{|x|^{(\alpha+1)p}}dx$$
\begin{equation}\label{aremterm10}
\geq c_{p}|\sigma|\int_{0}^{\infty}\left|\frac{d}{dr}\widetilde{g}\right|^{p}r^{p-1}dr=
c_{p}\int_{\mathbb{G}}\left|\frac{d}{dr}g\right|^{p}|x|^{p-Q}dx.
\end{equation}
By Theorem 3.1 in \cite{Ruzhansky-Suragan:critical} or Remark \ref{ScalHardyrem} with $\gamma=p$, we obtain
$$J(f)\geq c_{p}\int_{\mathbb{G}}|\mathcal{R}g|^{p}|x|^{p-Q}dx\geq c_{p}\left(\frac{p-1}{p}\right)^{p}\int_{\mathbb{G}}\frac{\left|g(x)-g(\frac{Rx}{|x|})\right|^{p}}{\left|
\log\frac{R}{|x|}\right|^{p}|x|^{Q}}dx$$
$$=c_{p}\left(\frac{p-1}{p}\right)^{p}\int_{\mathbb{G}}\frac{\left||x|^{\frac{Q-p-\alpha p}{p}}f(x)-R^{\frac{Q-p-\alpha p}{p}}f(\frac{Rx}{|x|})\right|^{p}}{\left|
\log\frac{R}{|x|}\right|^{p}|x|^{Q}}dx$$
for any $R>0$.
Here using $f(x)=\widetilde{f}(r)$, $r=|x|$, one calculates
$$J(f)\geq c_{p}\left(\frac{p-1}{p}\right)^{p}\int_{\mathbb{G}}\frac{\left|f(x)-R^{\frac{Q-p-\alpha p}{p}}\widetilde{f}(R)|x|^{-\frac{Q-p-\alpha p}{p}}\right|^{p}}{\left|
\log\frac{R}{|x|}\right|^{p}|x|^{(\alpha+1)p}}dx$$
$$=c_{p}\left(\frac{p-1}{p}\right)^{p}\int_{\mathbb{G}}\frac{\left|f(x)-c_{f}(R)|x|^{-\frac{Q-p-\alpha p}{p}}\right|^{p}}{\left|
\log\frac{R}{|x|}\right|^{p}|x|^{(\alpha+1)p}}dx,$$
\end{proof}
yielding \eqref{aremterm9}.

\section{Critical Hardy inequalities of logarithmic type and uncertainty principle}
\label{SEC:critHardy}
In this section, we present critical Hardy inequalities of logarithmic type on the homogeneous group
$\mathbb{G}$. In the abelian isotropic case, the following result was obtained in \cite{MOW15}. In the case $\gamma=p$ this result on the homogeneous group was proved in \cite{Ruzhansky-Suragan:critical}.
\begin{thm}\label{ScalHardy} Let $\mathbb{G}$ be a homogeneous group
of homogeneous dimension $Q$. Let $|\cdot|$ be a homogeneous quasi-norm. Let $1<\gamma<\infty$ and $\max\{1,\gamma-1\}<p<\infty$.
Then for all $f\in C_{0}^{\infty}(\mathbb{G}\backslash\{0\})$ and all $R>0$ we have
\begin{equation}\label{ScalHardy1}\left\|\frac{f-f_{R}}{|x|^{\frac{Q}{p}}
\left(\log\frac{R}{|x|}\right)^{\frac{\gamma}{p}}}\right\|_{L^{p}(\mathbb{G})}
\leq\frac{p}{\gamma-1}\left\|
\frac{\mathcal{R}f}{|x|^{\frac{Q-p}{p}}\left(\log\frac{R}{|x|}\right)^{\frac{\gamma-p}{p}}} \right\|_{L^{p}(\mathbb{G})},
\end{equation}
where $f_{R}(x)=f\left(R\frac{x}{|x|}\right)$, $\mathcal{R}$ is defined in \eqref{EQ:Euler}, and the constant $\frac{p}{\gamma-1}$ is optimal.
\end{thm}
\begin{proof}[Proof of Theorem \ref{ScalHardy}]
First, let us consider the integrals in \eqref{ScalHardy1} restricted to
$B(0, R)$. Introducing polar coordinates $(r,y)=(|x|, \frac{x}{\mid x\mid})\in (0,\infty)\times\wp$ on $\mathbb{G}$, where $\wp$ is the sphere as in \eqref{EQ:sphere}, and using \eqref{EQ:polar}, we have
$$\int_{B(0,R)}\frac{|f(x)-f_{R}(x)|^{p}}
{|x|^{Q}\left|\log\frac{R}{|x|}\right|^{\gamma}}dx$$
$$=\int_{0}^{R}\int_{\wp}\frac{|f(ry)-f(Ry)|^{p}}{r^{Q}\left(\log\frac{R}{r}\right)^{\gamma}}r^{Q-1}d\sigma(y)dr$$
$$=\int_{0}^{R}\frac{d}{dr}\left(\frac{1}{(\gamma-1)\left(\log\frac{R}{r}\right)^{\gamma-1}}\int_{\wp}|f(ry)-f(Ry)|^{p}d\sigma(y)\right)dr$$
$$-\frac{p}{\gamma-1}{\rm Re}\int_{0}^{R}\left(\log\frac{R}{r}\right)^{-\gamma+1}
\int_{\wp}|f(ry)-f(Ry)|^{p-2}(f(ry)-f(Ry))\overline{\frac{df(ry)}{dr}}d\sigma(y)dr$$
$$=-\frac{p}{\gamma-1}{\rm Re}\int_{0}^{R}\left(\log\frac{R}{r}\right)^{-\gamma+1}\int_{\wp}|f(ry)-f(Ry)|^{p-2}
(f(ry)-f(Ry))\overline{\frac{df(ry)}{dr}}d\sigma(y)dr,$$
where $p-\gamma+1>0$, so that the boundary term at $r=R$ vanishes due to inequalities
$$|f(ry)-f(Ry)|\leq C(R-r),$$
$$\log\frac{R}{r}\geq\frac{R-r}{R}.$$
Then, by the H\"{o}lder inequality, we get
$$\int_{0}^{R}\int_{\wp}\frac{|f(ry)-f(Ry)|^{p}}{r\left(\log\frac{R}{r}\right)^{\gamma}}d\sigma(y)dr$$
$$=-\frac{p}{\gamma-1}{\rm Re}\int_{0}^{R}\left(\log\frac{R}{r}\right)^{-\gamma+1}\int_{\wp}|f(ry)-f(Ry)|^{p-2}
(f(ry)-f(Ry))\overline{\frac{df(ry)}{dr}}d\sigma(y)dr$$
$$\leq\frac{p}{\gamma-1}\int_{0}^{R}\left(\log\frac{R}{r}\right)^{-\gamma+1}\int_{\wp}|f(ry)-f(Ry)|^{p-1}\left|\frac{df(ry)}{dr}\right| d\sigma(y)dr$$
$$\leq\frac{p}{\gamma-1}\left(\int_{0}^{R}\int_{\wp}\frac{|f(ry)-f(Ry)|^{p}}{r\left(\log\frac{R}{r}\right)
^{\gamma}}d\sigma(y)dr\right)^{\frac{p-1}{p}}$$
$$\times\left(\int_{0}^{R}\int_{\wp}r^{p-1}\left(\log\frac{R}{r}\right)^{p-\gamma}
\left|\frac{df(ry)}{dr}\right|^{p}d\sigma(y)dr\right)^{\frac{1}{p}}.$$
Thus we obtain
$$\left(\int_{B(0,R)}\frac{\left|f(x)-f_{R}(x)\right|^{p}}{|x|^{Q}\left|\log\frac{R}{|x|}\right|^{\gamma}}
dx\right)^{\frac{1}{p}}$$
\begin{equation}\label{ScalHardy2}\leq\frac{p}{\gamma-1}\left(\int_{B(0,R)}|x|^{p-Q}\left|\log\frac{R}{|x|}\right|^{p-\gamma}
\left|\mathcal{R}f(x)\right|^{p}dx\right)^{\frac{1}{p}}.
\end{equation}
Similarly, we have
$$\left(\int_{B^{c}(0,R)}\frac{\left|f(x)-f_{R}(x)\right|^{p}}{|x|^{Q}\left|\log\frac{R}{|x|}\right|^{\gamma}}
dx\right)^{\frac{1}{p}}$$
\begin{equation}\label{ScalHardy3}\leq\frac{p}{\gamma-1}\left(\int_{B^{c}(0,R)}|x|^{p-Q}\left|\log\frac{R}{|x|}\right|
^{p-\gamma}\left|\mathcal{R}f(x)\right|^{p}dx\right)^{\frac{1}{p}}.
\end{equation}
The inequalities \eqref{ScalHardy2} and \eqref{ScalHardy3} imply \eqref{ScalHardy1}.

Now let us prove the optimality of the constant $\frac{p}{\gamma-1}$ in \eqref{ScalHardy1}. The inequality \eqref{ScalHardy1} gives that
\begin{equation}\label{optim1}
\left(\int_{B(0,R)}\frac{|f(x)|^{p}}{|x|^{Q}\left|\log\frac{R}{|x|}\right|^{\gamma}}\right)^{\frac{1}{p}}
\leq \frac{p}{\gamma-1}\left(\int_{B(0,R)}|x|^{p-Q}\left|\log\frac{R}{|x|}\right|^{p-\gamma}|\mathcal{R}f(x)|^{p}dx\right)^{\frac{1}{p}}.
\end{equation}
It is enough to prove the optimality of the constant $\frac{p}{\gamma-1}$ in \eqref{optim1}. As in the abelian case (see \cite[Section 3]{MOW15}), we define the following sequence of functions
$$f_{k}(x):=\begin{cases}
(\log(kR))^{\frac{\gamma-1}{p}}, \;\;\;{\rm when}\;\;\;|x|\leq\frac{1}{k},\\
(\log\frac{R}{|x|})^{\frac{\gamma-1}{p}}, \;\;\;{\rm when}\;\;\;\frac{1}{k}\leq|x|\leq\frac{R}{2},\\
\frac{2}{R}(\log2)^{\frac{\gamma-1}{p}}(R-|x|), \;\;\;{\rm when}\;\;\;\frac{R}{2}\leq|x|\leq R
\end{cases}$$
for large $k\in \mathbb{N}$. Letting $\widetilde{f}_{k}(r):=f_{k}(x)$ with $r=|x|\geq0$, we get
$$\frac{d}{dr}\widetilde{f}_{k}(r)=\begin{cases}
0, \;\;\;{\rm when}\;\;\;r<\frac{1}{k},\\
-\frac{\gamma-1}{p}r^{-1}(\log\frac{R}{r})^{\frac{\gamma-1}{p}-1}, \;\;\;{\rm when}\;\;\;\frac{1}{k}<r<\frac{R}{2},\\
-\frac{2}{R}(\log2)^{\frac{\gamma-1}{p}}, \;\;\;{\rm when}\;\;\;\frac{R}{2}<r<R.
\end{cases}$$
Denoting by $|\sigma|$ the $Q-1$ dimensional surface measure of the unit sphere,
by a direct calculation one has
\begin{multline*}
\int_{B(0,R)}|x|^{p-Q}
\left|\log\frac{R}{|x|}\right|^{p-\gamma}\left|\mathcal{R}f_{k}(x)\right|
^{p}dx
=|\sigma|\int_{0}^{R}r^{p-1}\left|\log\frac{R}{r}\right|^{p-\gamma}
\left|\frac{d}{dr}\widetilde{f}_{k}(r)\right|^{p}dr\\
=|\sigma|\left(\frac{\gamma-1}{p}\right)^{p}\int_{\frac{1}{k}}^{\frac{R}{2}}r^{-1}\left(\log\frac{ R}{r}\right)^{-1}dr
+|\sigma|(\log2)^{\gamma-1}\left(\frac{2}{R}\right)^{p}
\int_{\frac{R}{2}}^{R}r^{p-1}\left(\log
\frac{R}{r}\right)^{p-\gamma}dr\end{multline*}

\begin{equation} \label{optim2}
=|\sigma|\left(\frac{\gamma-1}{p}\right)^{p}\left((\log(\log kR))-\log(\log 2)\right)+C_{\gamma,p},
\end{equation}
where $$C_{\gamma,p}:=2^{p}(\log2)^{\gamma-1}|\sigma|\int_{0}^{\log2}s^{p-\gamma}
e^{-ps}ds.
$$
Since $p-\gamma+1>0$, we get $C_{\gamma,q}<+\infty$. On the other hand, we see
$$
\int_{B(0,R)}\frac{|f_{k}(x)|^{p}}{|x|^{Q}\left|\log\frac{R}{|x|}\right|^{\gamma}}dx
=|\sigma|\int_{0}^{R}\frac{|\widetilde{f}_{k}(r)|^{p}}{r\left|\log\frac{R}{r}\right|^{\gamma}}dr
$$$$=|\sigma|(\log(kR))^{\gamma-1}\int_{0}^{\frac{1}{k}}r^{-1}
\left(\log\frac{R}{r}\right)^{-\gamma}dr
+|\sigma|\int_{\frac{1}{k}}^{\frac{R}{2}}r^{-1}
\left(\log\frac{R}{r}\right)^{-1}dr$$$$+|\sigma|(\log2)^{\gamma-1}\left(\frac{2}{R}\right)^{p}
\int_{\frac{R}{2}}^{R}r^{-1}
(R-r)^{p}\left(\log\frac{R}{r}\right)^{-\gamma}dr
$$
\begin{equation}\label{optim3}
=\frac{|\sigma|}{\gamma-1}+|\sigma|(\log(\log(kR))-\log(\log(2)))
+C_{R,\gamma,p},
\end{equation}
where
$$
C_{R,\gamma,p}:=(\log2)^{\gamma-1}\left(\frac{2}{R}\right)^{p}|\sigma|\int_{\frac{R}{2}}^{R}r^{-1}
(R-r)^{p}\left(\log\frac{R}{r}\right)^{-\gamma}dr.
$$
The inequality $\log\frac{R}{r}\geq \frac{R-r}{R}$ for all $r\leq R$ and the assumption $p-\gamma>-1$, imply $C_{R,\gamma,p}<+\infty$. Then, by \eqref{optim2} and \eqref{optim3}, we have
\begin{multline*}
\left(\int_{B(0,R)}|x|^{p-Q}
\left|\log\frac{R}{|x|}\right|^{p-\gamma}\left|\mathcal{R}f_{k}(x)\right|
^{p}dx\right)\\ \times \left(\int_{B(0,R)}\frac{|f_{k}(x)|^{p}}{|x|^{Q}\left|\log\frac{R}{|x|}\right|^{\gamma}}dx\right)^{-1}
\rightarrow \left(\frac{\gamma-1}{p}\right)^{p}
\end{multline*}
as $k \rightarrow \infty$, which implies that the constant $\frac{p}{\gamma-1}$ in \eqref{optim1} is optimal.
\end{proof}
\begin{cor}[{\rm Uncertainty type principle on $\mathbb{G}$}]\label{ScalHardycor} Let $1<p<\infty$ and $q>1$ be such that
$\frac{1}{p}+\frac{1}{q}=\frac{1}{2}$. Let $1<\gamma<\infty$ and $\max\{1,\gamma-1\}<p<\infty$. Then for any $R>0$ and $f\in C_{0}^{\infty}(\mathbb{G}\backslash\{0\})$ we have
\begin{equation}\label{ScalHardycor1}
\left\|\frac{\mathcal{R}f}{|x|^{\frac{Q-p}{p}}\left(\log\frac{R}{|x|}\right)^{\frac{\gamma-p}{p}}} \right\|_{L^{p}(\mathbb{G})}\|f\|_{L^{q}(\mathbb{G})}
\geq\frac{\gamma-1}{p}\left\|\frac{f(f-f_{R})}
{|x|^{\frac{Q}{p}}\left(\log\frac{R}{|x|}\right)^{\frac{\gamma}{p}}} \right\|_{L^{2}(\mathbb{G})}.
\end{equation}
Moreover,
\begin{multline}\label{ScalHardycor2}
\left\|\frac{\mathcal{R}f}{|x|^{\frac{Q-p}{p}}\left(\log\frac{R}{|x|}\right)^{\frac{\gamma-p}{p}}} \right\|_{L^{p}(\mathbb{G})}\left\|\frac{f-f_{R}}
{|x|^{\frac{Q}{p'}}\left(\log\frac{R}{|x|}\right)^{2-\frac{\gamma}{p}}}
\right\|_{L^{p'}(\mathbb{G})}\geq\frac{\gamma-1}{p}\left\|\frac{f-f_{R}}
{|x|^{\frac{Q}{2}}\log\frac{R}{|x|}}\right\|^{2}_{L^{2}(\mathbb{G})}
\end{multline}
holds for $\frac{1}{p}+\frac{1}{p'}=1$.
\end{cor}
\begin{proof}[Proof of Corollary \ref{ScalHardycor}]
By \eqref{ScalHardy1}, we have
$$\left\|\frac{\mathcal{R}f}{|x|^{\frac{Q-p}{p}}\left(\log\frac{R}{|x|}\right)^{\frac{\gamma-p}{p}}} \right\|_{L^{p}(\mathbb{G})}\|f\|_{L^{q}(\mathbb{G})}\geq\frac{\gamma-1}{p}
\left\|\frac{f-f_{R}}{|x|^{\frac{Q}{p}}\left(\log\frac{R}{|x|}\right)^{\frac{\gamma}{p}}}\right\|_{L^{p}(\mathbb{G})}
\|f\|_{L^{q}(\mathbb{G})}$$
$$=\frac{\gamma-1}{p}
\left(\int_{\mathbb{G}}\left|\frac{f(x)-f_{R}(x)}{|x|^{\frac{Q}{p}}\left(\log\frac{R}{|x|}\right)^{\frac{\gamma}{p}}}
\right|^{2\frac{p}{2}}dx\right)^{\frac{1}{2}\frac{2}{p}}
\left(\int_{\mathbb{G}}|f(x)|^{2\frac{q}{2}}dx\right)^{\frac{1}{2}\frac{2}{q}},$$
and using the H\"{o}lder inequality, we obtain
$$\left\||x|^{\frac{p-Q}{p}}\left(\log\frac{R}{|x|}\right)^{\frac{p-\gamma}{p}}
\mathcal{R}f \right\|_{L^{p}(\mathbb{G})}\|f\|_{L^{q}(\mathbb{G})}$$
$$\geq\frac{\gamma-1}{p}\left(\int_{\mathbb{G}}\left|\frac{f(x)(f(x)-f_{R}(x))}
{|x|^{\frac{Q}{p}}\left(\log\frac{R}{|x|}\right)^{\frac{\gamma}{p}}}\right|^{2}dx\right)^{\frac{1}{2}}
=\frac{\gamma-1}{p}\left\|\frac{f(f-f_{R})}
{|x|^{\frac{Q}{p}}\left(\log\frac{R}{|x|}\right)^{\frac{\gamma}{p}}} \right\|_{L^{2}(\mathbb{G})}.$$
Similarly, one can prove \eqref{ScalHardycor2}.
\end{proof}
\begin{rem}\label{ScalHardyrem}
When $\gamma=p$, Theorem \ref{ScalHardy} gives \cite[Theorem 3.1]{Ruzhansky-Suragan:critical}:
\begin{equation}\label{ScalHardy4}
\left\|\frac{f-f_{R}}{|x|^{\frac{Q}{p}}\log\frac{R}{|x|}}\right\|_{L^{p}(\mathbb{G})}
\leq\frac{p}{p-1}\left\||x|^{\frac{p-Q}{p}}
\mathcal{R}f \right\|_{L^{p}(\mathbb{G})}, \;\;1<p<\infty,
\end{equation}
for all $R>0$.
\end{rem}

\section{Critical and subcritical Hardy inequalities}
\label{SEC:crit_subcrit_con}

In this section, we study the relation between the critical and the subcritical Hardy inequalities on homogeneous groups.

\begin{prop}\label{crit_subcrit_pr2}
Let $\mathbb{G}$ be a homogeneous group
of homogeneous dimension $Q\geq 3$ and $Q\geq m+1$, $m\geq2$. Let $|\cdot|$ be a homogeneous quasi-norm. Then for any nonnegative radial function $g\in C_{0}^{1}(B^{m}(0,R)\backslash\{0\})$, there exists a nonnegative radial function $f\in C_{0}^{1}(B^{Q}(0,1)\backslash\{0\})$ such that
$$\int_{B^{Q}(0,1)}|\mathcal{R}f|^{m}dx-\left(\frac{Q-m}{m}\right)^{m}\int_{B^{Q}(0,1)}\frac{|f|^{m}}{|x|^{m}}dx$$
\begin{equation}\label{crit_subcrit3}
=\frac{|\sigma|}
{|\widetilde{\sigma}|}\left(\frac{Q-m}{m-1}\right)^{m-1}\left(\int_{B^{m}(0,R)}|\mathcal{R}g|^{m}dz-
\left(\frac{m-1}{m}\right)^{m}\int_{B^{m}(0,R)}\frac{|g|^{m}}{|z|^{m}\left(\log\frac{Re}{|z|}\right)^{m}}dz\right)
\end{equation}
holds true, where $\mathcal{R}$ is defined in \eqref{EQ:Euler}, $|\sigma|$ and $|\widetilde{\sigma}|$ are $Q-1$ and $m-1$ dimensional surface measure of the unit sphere, respectively.
\end{prop}
\begin{proof}[Proof of Proposition \ref{crit_subcrit_pr2}] Let $r=|x|$, $x\in\mathbb{G}$ and $s=|z|$, $z\in\mathbb{\widetilde{G}}$, where $\mathbb{\widetilde{G}}$ is a homogeneous group
of homogeneous dimension $m$. Let us define a radial function $f=f(x)\in C_{0}^{1}(B^{Q}(0,1)\backslash\{0\})$ for a nonnegative radial function $g=g(z)\in C_{0}^{1}(B^{m}(0,R)\backslash\{0\})$:
\begin{equation}\label{crit_subcrit4}
f(r)=g(s(r)),
\end{equation}
where $s(r)=R\exp(1-r^{-\frac{Q-m}{m-1}})$, that is,
$$r^{-\frac{Q-m}{m-1}}=\log\frac{Re}{s}, \;s'(r)=\frac{Q-m}{m-1}r^{-\frac{Q-m}{m-1}-1}s(r).$$

Here we see that $s'(r)>0$ for $r\in[0,1]$ and $s(0)=0$, $s(1)=R$. Since $g(s)\equiv0$ near $s=R$, we also note that $f\equiv0$ near $r=1$. Then a direct calculation shows
$$\int_{B^{Q}(0,1)}|\mathcal{R}f|^{m}dx-\left(\frac{Q-m}{m}\right)^{m}\int_{B^{Q}(0,1)}\frac{|f|^{m}}{|x|^{m}}dx$$
$$=|\sigma|\int_{0}^{1}|f'(r)|^{m}r^{Q-1}dr-\left(\frac{Q-m}{m}\right)^{m}|\sigma|\int_{0}^{1}f^{m}(r)r^{Q-m-1}dr$$
$$=|\sigma|\int_{0}^{R}|g'(s)s'(r(s))|^{m}r^{Q-1}(s)\frac{ds}{s'(r(s))}-\left(\frac{Q-m}{m}\right)^{m}|\sigma|\int_{0}^{R}g^{m}(s)
r^{Q-m-1}(s)\frac{ds}{s'(r(s))}$$
$$=|\sigma|\left(\frac{Q-m}{m-1}\right)^{m-1}\int_{0}^{R}|g'(s)|^{m}s^{m-1}ds-\left(\frac{Q-m}{m}\right)^{m}\frac{m-1}{Q-m}
|\sigma|\int_{0}^{R}\frac{g^{m}(s)}{s\left(\log\frac{Re}{s}\right)^{m}}ds$$
$$=\frac{|\sigma|}{|\widetilde{\sigma}|}\left(\frac{Q-m}{m-1}\right)^{m-1}\left(\int_{B^{m}(0,R)}|\mathcal{R}g|^{m}dz-
\left(\frac{m-1}{m}\right)^{m}\int_{B^{m}(0,R)}\frac{|g|^{m}}{|z|^{m}\left(\log\frac{Re}{|z|}\right)^{m}}dz\right),$$
yielding \eqref{crit_subcrit3}.
\end{proof}

\section{Extended Caffarelli-Kohn-Nirenberg inequalities}\label{SEC:CKN}

In this section, we introduce new Caffarelli-Kohn-Nirenberg type inequalities in the Euclidean setting of $\Rn$ as well as on homogeneous groups. For the convenience of the reader we recall Theorem \ref{THM:CKN-i} and then also explain how it implies Theorem \ref{clas_CKN-2}:

\begin{thm}\label{CKN_thm}
Let $\mathbb{G}$ be a homogeneous group
of homogeneous dimension $Q$. Let $|\cdot|$ be a homogeneous quasi-norm. Let $1<p,q<\infty$, $0<r<\infty$ with $p+q\geq r$ and $\delta\in[0,1]\cap\left[\frac{r-q}{r},\frac{p}{r}\right]$ and $a$, $b$, $c\in\mathbb{R}$. Assume that $\frac{\delta r}{p}+\frac{(1-\delta)r}{q}=1$ and $c=\delta(a-1)+b(1-\delta)$. Then we have the following Caffarelli-Kohn-Nirenberg type inequalities for all $f\in C_{0}^{\infty}(\mathbb{G}\backslash\{0\})$:

If $Q\neq p(1-a)$, then
\begin{equation}\label{CKN_thm2}
\||x|^{c}f\|_{L^{r}(\mathbb{G})}
\leq \left|\frac{p}{Q-p(1-a)}\right|^{\delta} \left\||x|^{a}\mathcal{R}f\right\|^{\delta}_{L^{p}(\mathbb{G})}
\left\||x|^{b}f\right\|^{1-\delta}_{L^{q}(\mathbb{G})}.
\end{equation}
If $Q=p(1-a)$, then
\begin{equation}\label{CKN_thm1}
\left\||x|^{c}f\right\|_{L^{r}(\mathbb{G})}
\leq p^{\delta} \left\||x|^{a}\log|x|\mathcal{R}f\right\|^{\delta}_{L^{p}(\mathbb{G})}
\left\||x|^{b}f\right\|^{1-\delta}_{L^{q}(\mathbb{G})}.
\end{equation}
The constant in the inequality \eqref{CKN_thm2} is sharp for $p=q$ with $a-b=1$ or $p\neq q$ with $p(1-a)+bq\neq0$. Moreover, the constants in \eqref{CKN_thm2} and \eqref{CKN_thm1} are sharp for $\delta=0$ or $\delta=1$. Here $\mathcal{R}:=\frac{d}{d|x|}$ is the radial derivative.
\end{thm}
\begin{rem} Our conditions $\frac{\delta r}{p}+\frac{(1-\delta)r}{q}=1$ and $c=\delta(a-1)+b(1-\delta)$ imply the condition \eqref{clas_CKN2} of Theorem \ref{clas_CKN}, and in our case $a-d=1$.
\end{rem}
\begin{rem}\label{CKN_rem1}
In the abelian case $\mathbb{G}=(\Rn,+)$ and $Q=n$, \eqref{CKN_thm1} implies a new type of Caffarelli-Kohn-Nirenberg inequality for any quasi-norm on $\Rn$: Let $1<p,q<\infty$, $0<r<\infty$ with $p+q\geq r$ and $\delta\in[0,1]\cap\left[\frac{r-q}{r},\frac{p}{r}\right]$ and $a$, $b$, $c\in\mathbb{R}$. Assume that $\frac{\delta r}{p}+\frac{(1-\delta)r}{q}=1$, $n=p(1-a)$ and $c=\delta(a-1)+b(1-\delta)$. Then we have the Caffarelli-Kohn-Nirenberg type inequality for any function $f\in C_{0}^{\infty}(\mathbb{R}^{n}\backslash\{0\})$ and for any homogeneous quasi-norm $|\cdot|$:
\begin{equation}\label{CKN_rem2}
\left\||x|^{c}f\right\|_{L^{r}(\Rn)}
\leq p^{\delta} \left\||x|^{a}\log|x|\left(\frac{x}{|x|}\cdot\nabla f\right)\right\|^{\delta}_{L^{p}(\Rn)}
\left\||x|^{b}f\right\|^{1-\delta}_{L^{q}(\Rn)}.
\end{equation}
By the Schwarz inequality with the standard Euclidean distance given by $|x|=\sqrt{x_{1}^{2}+x_{2}^{2}+...+x_{n}^{2}}$, we obtain the Euclidean form of the Caffarelli-Kohn-Nirenberg type inequality:
\begin{equation}\label{CKN_rem3}
\left\||x|^{c}f\right\|_{L^{r}(\Rn)}
\leq p^{\delta} \left\||x|^{a}\log|x|\nabla f\right\|^{\delta}_{L^{p}(\Rn)}
\left\||x|^{b}f\right\|^{1-\delta}_{L^{q}(\Rn)},
\end{equation}
where $\nabla$ is the standard gradient in $\Rn$.
Similarly, we write the inequality \eqref{CKN_thm2} in the abelian case: Let $1<p,q<\infty$, $0<r<\infty$ with $p+q\geq r$ and $\delta\in[0,1]\cap\left[\frac{r-q}{r},\frac{p}{r}\right]$ and $a$, $b$, $c\in\mathbb{R}$. Assume that $\frac{\delta r}{p}+\frac{(1-\delta)r}{q}=1$, $n\neq p(1-a)$ and $c=\delta(a-1)+b(1-\delta)$. Then we have Caffarelli-Kohn-Nirenberg type inequality for any function $f\in C_{0}^{\infty}(\mathbb{G}\backslash\{0\})$ and for any homogeneous quasi-norm $|\cdot|$:
\begin{equation}
\||x|^{c}f\|_{L^{r}(\Rn)}
\leq \left|\frac{p}{n-p(1-a)}\right|^{\delta} \left\||x|^{a}\left(\frac{x}{|x|}\cdot\nabla f\right)\right\|^{\delta}_{L^{p}(\Rn)}
\left\||x|^{b}f\right\|^{1-\delta}_{L^{q}(\Rn)}.
\end{equation}
Then, using the Schwarz inequality with the standard Euclidean distance given by $|x|=\sqrt{x_{1}^{2}+x_{2}^{2}+...+x_{n}^{2}}$, we obtain the Euclidean form of the Caffarelli-Kohn-Nirenberg type inequality:
\begin{equation}\label{CKN_rem3_1}
\||x|^{c}f\|_{L^{r}(\Rn)}
\leq \left|\frac{p}{n-p(1-a)}\right|^{\delta} \left\||x|^{a}\nabla f\right\|^{\delta}_{L^{p}(\Rn)}
\left\||x|^{b}f\right\|^{1-\delta}_{L^{q}(\Rn)}.
\end{equation}
Note that if
\begin{equation}\label{CKNrem1}\frac{1}{p}+\frac{a}{n}>0, \;\frac{1}{q}+\frac{b}{n}>0 \;\;{\rm and}\;\; \frac{1}{r}+\frac{c}{n}>0
\end{equation} hold, then the inequality \eqref{CKN_rem3_1} is contained in the family of Caffarelli-Kohn-Nirenberg inequalities \cite{CKN84}. In this case, if we require $p=q$ with $a-b=1$ or $p\neq q$ with $p(1-a)+bq\neq0$, then we obtain the inequality \eqref{CKN_rem3_1} with the sharp constant. Moreover, the constants $\left|\frac{p}{n-p(1-a)}\right|^{\delta}$ and $p^{\delta}$ are sharp for $\delta=0$ or $\delta=1$. If \eqref{CKNrem1} is not satisfied, then the inequality \eqref{CKN_rem3_1} is not covered by Theorem \ref{clas_CKN} because condition \eqref{clas_CKN0} fails. So we obtain a new range of the  Caffarelli-Kohn-Nirenberg inequality \cite{CKN84}.

Thus, the inequalities \eqref{CKN_rem3} and \eqref{CKN_rem3_1} are new already in the abelian case and, moreover, \eqref{CKN_thm2} and \eqref{CKN_thm1} hold for any choice of homogeneous quasi-norm.
\end{rem}

The proof of Theorem \ref{CKN_thm} will be based on the following family of weighted Hardy inequalities that was obtained in \cite[Theorem 3.4]{RSY16}, where $\mathbb{E}=|x|\mathcal{R}$ is the Euler operator.

\begin{thm}[\cite{RSY16}]\label{L_p_weighted_th}
Let $\mathbb{G}$ be a homogeneous group
of homogeneous dimension $Q$ and let $\alpha\in \mathbb{R}$.
Then for all complex-valued functions $f\in C^{\infty}_{0}(\mathbb{G}\backslash\{0\}),$ $1<p<\infty,$ and any homogeneous quasi-norm $|\cdot|$ on $\mathbb{G}$ for $\alpha p \neq Q$ we have
\begin{equation}\label{L_p_weighted}
\left\|\frac{f}{|x|^{\alpha}}\right\|_{L^{p}(\mathbb{G})}\leq
\left|\frac{p}{Q-\alpha p}\right|\left\|\frac{1}{|x|^{\alpha}}\mathbb{E} f\right\|_{L^{p}(\mathbb{G})}.
\end{equation}
If $\alpha p\neq Q$ then the constant $\left|\frac{p}{Q-\alpha p}\right|$ is sharp.
For $\alpha p=Q$ we have
\begin{equation}\label{L_p_weighted_log}
\left\|\frac{f}{|x|^{\frac{Q}{p}}}\right\|_{L^{p}(\mathbb{G})}\leq
p\left\|\frac{\log|x|}{|x|^{\frac{Q}{p}}}\mathbb{E} f\right\|_{L^{p}(\mathbb{G})},
\end{equation}
where the constant $p$ is sharp.
\end{thm}

We briefly recall its proof for the convenience of the reader but also since it will be useful in our argument.

\begin{proof}[Proof of Theorem \ref{L_p_weighted_th}]
Using integration by parts, for $\alpha p \neq Q$ we obtain
\begin{multline*}
\int_{\mathbb{G}}\frac{|f(x)|^{p}}{|x|^{\alpha p}}dx=\int_{0}^{\infty}\int_{\wp}|f(ry)|^{p}r^{Q-1-\alpha p}d\sigma(y)dr\\
=-\frac{p}{Q-\alpha p}\int_{0}^{\infty} r^{Q-\alpha p} {\rm Re} \int_{\wp}|f(ry)|^{p-2} f(ry) \overline{\frac{df(ry)}{dr}}d\sigma(y)dr\\
\leq \left|\frac{p}{Q-\alpha p}\right|\int_{\mathbb{G}}\frac{|\mathbb{E}f(x)||f(x)|^{p-1}}{|x|^{\alpha p}}dx=
\left|\frac{p}{Q-\alpha p}\right|\int_{\mathbb{G}}\frac{|\mathbb{E}f(x)||f(x)|^{p-1}}{|x|^{\alpha+\alpha (p-1)}}dx.
\end{multline*}
By H\"{o}lder's inequality, it follows that
$$
\int_{\mathbb{G}}\frac{|f(x)|^{p}}{|x|^{\alpha p}}dx\leq \left|\frac{p}{Q-\alpha p}\right|\left(\int_{\mathbb{G}}\frac{|\mathbb{E}f(x)|^{p}}{|x|^{\alpha p}}dx\right)
^{\frac{1}{p}}\left(\int_{\mathbb{G}}\frac{|f(x)|^{p}}{|x|^{\alpha p}}dx\right)^{\frac{p-1}{p}},
$$
which gives \eqref{L_p_weighted}.

Now we show the sharpness of the constant. We need to check the equality
condition in above H\"older's inequality.
Let us consider the function
\begin{equation}\label{2}
g(x)=\frac{1}{|x|^{C}},
\end{equation}
where $C\in\mathbb{R}, C\neq 0$ and $\alpha p\neq Q$. Then by a direct calculation we obtain
\begin{equation}\label{Holder_eq1}
\left|\frac{1}{C}\right|^{p}\left(\frac{|\mathbb{E}g(x)|}{|x|^{\alpha }}\right)^{p}=\left(\frac{|g(x)|^{p-1}}
{|x|^{\alpha (p-1)}}\right)^{\frac{p}{p-1}},
\end{equation}
which satisfies the equality condition in H\"older's inequality.
This gives the sharpness of the constant $\left|\frac{p}{Q-\alpha p}\right|$ in \eqref{L_p_weighted}.

Now let us prove \eqref{L_p_weighted_log}. Using integration by parts, we have
\begin{multline*}
\int_{\mathbb{G}}\frac{|f(x)|^{p}}{|x|^{Q}}dx=\int_{0}^{\infty}\int_{\wp}|f(ry)|^{p}r^{Q-1-Q}d\sigma(y)dr\\
=-p\int_{0}^{\infty} \log r {\rm Re} \int_{\wp}|f(ry)|^{p-2} f(ry) \overline{\frac{df(ry)}{dr}}d\sigma(y)dr\\
\leq p \int_{\mathbb{G}}\frac{|\mathbb{E}f(x)||f(x)|^{p-1}}{|x|^{Q}}|\log|x||dx=
p\int_{\mathbb{G}}\frac{|\mathbb{E}f(x)|\log|x|||}{|x|^{\frac{Q}{p}}}\frac{|f(x)|^{p-1}}{|x|^{\frac{Q(p-1)}{p}}}dx.
\end{multline*}
By H\"{o}lder's inequality, it follows that
$$
\int_{\mathbb{G}}\frac{|f(x)|^{p}}{|x|^{Q}}dx\leq p\left(\int_{\mathbb{G}}\frac{|\mathbb{E}f(x)|^{p}|\log|x||^{p}}{|x|^{Q}}dx\right)
^{\frac{1}{p}}\left(\int_{\mathbb{G}}\frac{|f(x)|^{p}}{|x|^{Q}}dx\right)^{\frac{p-1}{p}},
$$
which gives \eqref{L_p_weighted_log}.

Now we show the sharpness of the constant. We need to check the equality
condition in above H\"older's inequality.
Let us consider the function
$$h(x)=(\log|x|)^{C},$$
where $C\in\mathbb{R}$ and $C\neq 0$.
Then by a direct calculation we obtain
\begin{equation}\label{Holder_eq2}
\left|\frac{1}{C}\right|^{p}\left(\frac{|\mathbb{E}h(x)||\log|x||}{|x|^{\frac{Q}{p}}}\right)^{p}=\left(\frac{|h(x)|^{p-1}}
{|x|^{\frac{Q (p-1)}{p}}}\right)^{\frac{p}{p-1}},
\end{equation}
which satisfies the equality condition in H\"older's inequality.
This gives the sharpness of the constant $p$ in \eqref{L_p_weighted_log}.
\end{proof}

We are now ready to prove Theorem \ref{CKN_thm}.

\begin{proof}[Proof of Theorem \ref{CKN_thm}] {\bf Case $\delta=0$}. In this case, we have $q=r$ and $b=c$ by $\frac{\delta r}{p}+\frac{(1-\delta)r}{q}=1$ and $c=\delta(a-1)+b(1-\delta)$, respectively. Then, the inequalities \eqref{CKN_thm2} and \eqref{CKN_thm1} are equivalent to the trivial estimate
$$
\||x|^{b}f\|_{L^{q}(\mathbb{G})}
\leq \left\||x|^{b}f\right\|_{L^{q}(\mathbb{G})}.
$$
{\bf Case $\delta=1$}. Notice that in this case, $p=r$ and $a-1=c$. By Theorem \ref{L_p_weighted_th}, we have for $Q+pc=Q+p(a-1)\neq0$ the inequality
$$\||x|^{c}f\|_{L^{r}(\mathbb{G})}\leq \left|\frac{p}{Q+pc}\right|\||x|^{c}\mathbb{E}f\|_{L^{r}(\mathbb{G})},$$
where $\mathbb{E}=|x|\mathcal{R}$ is the Euler operator. Taking into account this, we get
$$\||x|^{c}f\|_{L^{r}(\mathbb{G})}\leq \left|\frac{p}{Q+pc}\right|\||x|^{c+1}\mathcal{R}f\|_{L^{r}(\mathbb{G})}$$
$$=\left|\frac{p}{Q-p(1-a)}\right|\||x|^{a}\mathcal{R}f\|_{L^{p}(\mathbb{G})},$$
which implies \eqref{CKN_thm2}.
For $Q+pc=Q+p(a-1)=0$ by Theorem \ref{L_p_weighted_th} we obtain
$$\||x|^{c}f\|_{L^{r}(\mathbb{G})}\leq p\||x|^{c}\log|x|\mathbb{E}f\|_{L^{r}(\mathbb{G})}=p\||x|^{c+1}\log|x|\mathcal{R}f\|_{L^{r}(\mathbb{G})}$$
$$=p\||x|^{a}\log|x|\mathcal{R}f\|_{L^{p}(\mathbb{G})}$$
which gives \eqref{CKN_thm1}. In this case, the constants in \eqref{CKN_thm2} and \eqref{CKN_thm1} are sharp, since the constants in Theorem \ref{L_p_weighted_th} are sharp.

{\bf Case $\delta\in(0,1)\cap\left[\frac{r-q}{r},\frac{p}{r}\right]$}.
Taking into account $c=\delta(a-1)+b(1-\delta)$,  a direct calculation gives $$\||x|^{c}f\|_{L^{r}(\mathbb{G})}=
\left(\int_{\mathbb{G}}|x|^{cr}|f(x)|^{r}dx\right)^{\frac{1}{r}}
=\left(\int_{\mathbb{G}}\frac{|f(x)|^{\delta r}}{|x|^{\delta r (1-a)}}\cdot \frac{|f(x)|^{(1-\delta)r}}{|x|^{-br(1-\delta)}}dx\right)^{\frac{1}{r}}.$$
Since we have $\delta\in(0,1)\cap\left[\frac{r-q}{r},\frac{p}{r}\right]$ and $p+q\geq r$, then by using H\"{o}lder's inequality for $\frac{\delta r}{p}+\frac{(1-\delta)r}{q}=1$, we obtain
$$\||x|^{c}f\|_{L^{r}(\mathbb{G})}
\leq \left(\int_{\mathbb{G}}\frac{|f(x)|^{p}}{|x|^{p(1-a)}}dx\right)^{\frac{\delta}{p}}
\left(\int_{\mathbb{G}}\frac{|f(x)|^{q}}{|x|^{-bq}}dx\right)^{\frac{1-\delta}{q}}$$
\begin{equation}\label{CKN_thm1_1}=\left\|\frac{f}{|x|^{1-a}}\right\|^{\delta}_{L^{p}(\mathbb{G})}
\left\|\frac{f}{|x|^{-b}}\right\|^{1-\delta}_{L^{q}(\mathbb{G})}.
\end{equation}
Here we note that when $p=q$ and $a-b=1$ H\"{o}lder's equality condition is held for any function. We also note that in the case $p\neq q$ the function
\begin{equation}\label{Holder_eq_2}
h(x)=|x|^{\frac{1}{(p-q)}\left(p(1-a)+bq\right)}
\end{equation} satisfies H\"{o}lder's equality condition:
$$\frac{|h|^{p}}{|x|^{p(1-a)}}=\frac{|h|^{q}}{|x|^{-bq}}.$$
If $Q\neq p(1-a)$, then by Theorem \ref{L_p_weighted_th}, we have
$$\left\|\frac{f}{|x|^{1-a}}\right\|^{\delta}_{L^{p}(\mathbb{G})}\leq
\left|\frac{p}{Q-p(1-a)}\right|^{\delta} \left\|\frac{\mathbb{E}f}{|x|^{1-a}}\right\|^{\delta}_{L^{p}(\mathbb{G})}$$
\begin{equation}\label{1}=
\left|\frac{p}{Q-p(1-a)}\right|^{\delta} \left\|\frac{\mathcal{R}f}{|x|^{-a}}\right\|^{\delta}_{L^{p}(\mathbb{G})}, \;\;1<p<\infty.
\end{equation}
Putting this in \eqref{CKN_thm1_1}, one has
$$\||x|^{c}f\|_{L^{r}(\mathbb{G})}\leq
\left|\frac{p}{Q-p(1-a)}\right|^{\delta} \left\|\frac{\mathcal{R}f}{|x|^{-a}}\right\|^{\delta}_{L^{p}(\mathbb{G})}
\left\|\frac{f}{|x|^{-b}}\right\|^{1-\delta}_{L^{q}(\mathbb{G})}.$$
We note that in the case $p=q$, $a-b=1$ H\"older's equality condition of the inequalities \eqref{CKN_thm1_1} and \eqref{1} holds true for $g(x)$ in \eqref{2}. Moreover, in the case $p\neq q$, $p(1-a)+bq\neq0$ H\"older's equality condition of the inequalities \eqref{CKN_thm1_1} and \eqref{1} holds true for $h(x)$ in \eqref{Holder_eq_2}. Therefore, the constant in \eqref{CKN_thm2} is sharp when $p=q$, $a-b=1$ or $p\neq q$, $p(1-a)+bq\neq0$.

Now let us consider the case $Q=p(1-a)$. Using Theorem \ref{L_p_weighted_th}, one has
$$\left\|\frac{f}{|x|^{1-a}}\right\|^{\delta}_{L^{p}(\mathbb{G})}\leq
p^{\delta} \left\|\frac{\log|x|}{|x|^{1-a}}\mathbb{E}f\right\|^{\delta}_{L^{p}(\mathbb{G})}, \;\;1<p<\infty.$$
Then, putting this in \eqref{CKN_thm1_1}, we obtain
$$\||x|^{c}f\|_{L^{r}(\mathbb{G})}\leq
p^{\delta} \left\|\frac{\log|x|}{|x|^{1-a}}\mathbb{E}f\right\|^{\delta}_{L^{p}(\mathbb{G})}
\left\|\frac{f}{|x|^{-b}}\right\|^{1-\delta}_{L^{q}(\mathbb{G})}
$$
$$=p^{\delta} \left\|\frac{\log|x|}{|x|^{-a}}\mathcal{R}f\right\|^{\delta}_{L^{p}(\mathbb{G})}
\left\|\frac{f}{|x|^{-b}}\right\|^{1-\delta}_{L^{q}(\mathbb{G})}.$$
\end{proof}

\section{$L^{p}$-Hardy inequalities with super weights}
\label{SEC:2}

We now discuss versions of Hardy inequalities with more general weights, that we call superweights.
The following is the main result of this section.

\begin{thm}\label{1}
Let $\mathbb{G}$ be a homogeneous group
of homogeneous dimension $Q$ and let $|\cdot|$ be a homogeneous quasi-norm on $\mathbb{G}$. Let $a,b>0$ and $1<p<\infty, Q\geq1$.
\begin{itemize}
\item[(i)] If $\alpha \beta>0$ and $pm\leq Q-p$, then for all $f\in C_{0}^{\infty}(\mathbb{G}\backslash\{0\})$, we have
\begin{equation}\label{Lpweighted1}
\frac{Q-pm-p}{p}
\left\|\frac{(a+b|x|^{\alpha})^{\frac{\beta}{p}}}{|x|^{m+1}}f\right\|_{L^{p}(\mathbb{G})}
\leq\left\|\frac{(a+b|x|^{\alpha})^{\frac{\beta}{p}}}{|x|^{m}}\mathcal{R}f\right\|_{L^{p}(\mathbb{G})} .
\end{equation}
If $Q\neq pm+p$, then the constant $\frac{Q-pm-p}{p}$ is sharp.

\item[(ii)] If $\alpha \beta<0$ and $pm-\alpha\beta\leq Q-p$, then for all $f\in C_{0}^{\infty}(\mathbb{G}\backslash\{0\})$, we have
\begin{equation}\label{Lpweighted2}
\frac{Q-pm+\alpha\beta-p}{p}
\left\|\frac{(a+b|x|^{\alpha})^{\frac{\beta}{p}}}{|x|^{m+1}}f\right\|_{L^{p}(\mathbb{G})}
\leq\left\|\frac{(a+b|x|^{\alpha})^{\frac{\beta}{p}}}{|x|^{m}}\mathcal{R}f\right\|_{L^{p}(\mathbb{G})}
.
\end{equation}
If $Q\neq pm+p-\alpha\beta$, then the constant $\frac{Q-pm+\alpha\beta-p}{p}$ is sharp.
\end{itemize}
\end{thm}
\begin{proof}[Proof of Theorem \ref{1}] We may assume that $Q\neq pm+p$ since for $Q=pm+p$ there is nothing
to prove. Introducing polar coordinates $(r,y)=(|x|, \frac{x}{\mid x\mid})\in (0,\infty)\times\wp$ on $\mathbb{G}$, where $\wp$ is the
unit quasi-sphere
\begin{equation}\label{EQ:sphere}
\wp:=\{x\in \mathbb{G}:\,|x|=1\},
\end{equation} and using the polar decomposition on homogeneous groups (see, for example, \cite{FS-Hardy} or \cite{FR}) and integrating by parts, we get
\begin{equation}\label{eqp}
\int_{\mathbb{G}}
\frac{(a+b|x|^{\alpha})^{\beta}}{|x|^{pm+p}}|f(x)|^{p}dx
=\int_{0}^{\infty}\int_{\wp}
\frac{(a+br^{\alpha})^{\beta}}{r^{pm+p}}|f(ry)|^{p} r^{Q-1}d\sigma(y)dr.
\end{equation}
(i) Since $a,b>0$, $\alpha \beta>0$ and $m<\frac{Q-p}{p}$ we obtain
$$
\int_{\mathbb{G}}
\frac{(a+b|x|^{\alpha})^{\beta}}{|x|^{pm+p}}|f(x)|^{p}dx$$
$$\leq\int_{0}^{\infty}\int_{\wp}
(a+br^{\alpha})^{\beta}r^{Q-1-pm-p}\left(\frac{\alpha\beta b r^{\alpha}}{(a+br^{\alpha})(Q-pm-p)}+1\right)|f(ry)|^{p} d\sigma(y)dr
$$
$$=\int_{0}^{\infty}\int_{\wp}
\frac{d}{dr}\left(\frac{(a+br^{\alpha})^{\beta}r^{Q-pm-p}}{Q-pm-p}\right)|f(ry)|^{p} d\sigma(y)dr$$
$$
=-\frac{p}{Q-pm-p}\int_{0}^{\infty}(a+br^{\alpha})^{\beta}r^{Q-pm-p}  \,{\rm Re}\int_{\wp}
|f(ry)|^{p-2} f(ry) \overline{\frac{df(ry)}{dr}}d\sigma(y)dr
$$
$$\leq \left|\frac{p}{Q-pm-p}\right|\int_{\mathbb{G}}\frac{(a+b|x|^{\alpha})^{\beta}|\mathcal{R}f(x)||f(x)|^{p-1}}{|x|^{pm+p-1}}dx
$$
$$=\frac{p}{Q-pm-p}
\int_{\mathbb{G}}\frac{(a+b|x|^{\alpha})^
{\frac{\beta(p-1)}{p}}|f(x)|^{p-1}}{|x|^{(m+1)(p-1)}}
\frac{(a+b|x|^{\alpha})^
{\frac{\beta}{p}}}{|x|^{m}}|\mathcal{R}f(x)|dx.$$
By H\"{o}lder's inequality, it follows that
$$
\int_{\mathbb{G}}
\frac{(a+b|x|^{\alpha})^{\beta}}{|x|^{pm+p}}|f(x)|^{p}dx$$
$$\leq\frac{p}{Q-pm-p}\left(\int_{\mathbb{G}}\frac{(a+b|x|^{\alpha})^{\beta}}{|x|^{pm+p}}|f(x)|^{p}dx\right)^\frac{p-1}{p}
\left(\int_{\mathbb{G}}\frac{(a+b|x|^{\alpha})^{\beta}}{|x|^{pm}}|\mathcal{R}f(x)|^{p}dx\right)^\frac{1}{p},
$$
which gives \eqref{Lpweighted1}.

Now we show the sharpness of the constant. We need to check the equality
condition in above H\"older's inequality.
Let us consider the function
$$g(x)=|x|^{C}, $$
where $C\in\mathbb{R}, C\neq 0$ and $Q\neq pm+p$. Then by a direct calculation we obtain
\begin{equation}\label{Holder_eq1}
\left|\frac{1}{C}\right|^{p}\left(\frac{(a+b|x|^{\alpha})^
{\frac{\beta}{p}}|\mathcal{R}g(x)|}{|x|^{m}}\right)^{p}=\left(\frac{(a+b|x|^{\alpha})^
{\frac{\beta(p-1)}{p}}|g(x)|^{p-1}}
{|x|^{(m+1) (p-1)}}\right)^{\frac{p}{p-1}},
\end{equation}
which satisfies the equality condition in H\"older's inequality.
This gives the sharpness of the constant $\frac{Q-pm-p}{p}$ in \eqref{Lpweighted1}.

Let us now prove Part (ii). Here we also assume that, $Q\neq pm+p-\alpha\beta$ since for $Q=pm+p-\alpha\beta$ there is nothing
to prove. Using the polar decomposition, we have
the equality \eqref{eqp}.
Since $\alpha \beta<0$ and $pm-\alpha\beta<Q-p$ we obtain
$$
\int_{\mathbb{G}}
\frac{(a+b|x|^{\alpha})^{\beta}}{|x|^{pm+p}}|f(x)|^{p}dx$$
$$\leq\int_{0}^{\infty}\int_{\wp}
(a+br^{\alpha})^{\beta}r^{Q-1-pm-p}\left(\frac{b r^{\alpha}}{a+br^{\alpha}}+\frac{a}{a+br^{\alpha}}
\cdot\frac{Q-pm-p}{Q-pm-p+\alpha\beta}\right)|f(ry)|^{p} d\sigma(y)dr
$$
$$=\int_{0}^{\infty}\int_{\wp}
\frac{(a+br^{\alpha})^{\beta}r^{Q-1-pm-p}}{Q-pm-p+\alpha\beta}
\left(\frac{\alpha \beta b r^{\alpha}}{a+br^{\alpha}}+Q-pm-p\right)|f(ry)|^{p} d\sigma(y)dr
$$
$$=\int_{0}^{\infty}\int_{\wp}
\frac{d}{dr}\left(\frac{(a+br^{\alpha})^{\beta}
	r^{Q-pm-p}}{Q-pm-p+\alpha\beta}\right)|f(ry)|^{p}
d\sigma(y)dr$$
$$
=-\frac{p}{Q-pm-p+\alpha\beta}\int_{0}^{\infty}(a+br^{\alpha})^{\beta}r^{Q-pm-p}  \,{\rm Re}\int_{\wp}
|f(ry)|^{p-2} f(ry) \overline{\frac{df(ry)}{dr}}d\sigma(y)dr
$$
$$\leq \left|\frac{p}{Q-pm-p+\alpha\beta}\right|\int_{\mathbb{G}}\frac{(a+b|x|^{\alpha})^{\beta}|\mathcal{R}f(x)||f(x)|^{p-1}}{|x|^{pm+p-1}}dx
$$
$$=\frac{p}{Q-pm-p+\alpha\beta}
\int_{\mathbb{G}}\frac{(a+b|x|^{\alpha})^
{\frac{\beta(p-1)}{p}}|f(x)|^{p-1}}{|x|^{(m+1)(p-1)}}
\frac{(a+b|x|^{\alpha})^
{\frac{\beta}{p}}}{|x|^{m}}|\mathcal{R}f(x)|dx.$$
By H\"{o}lder's inequality, it follows that
$$
\int_{\mathbb{G}}
\frac{(a+b|x|^{\alpha})^{\beta}}{|x|^{pm+p}}|f(x)|^{p}dx$$
$$\leq\frac{p}{Q-pm-p+\alpha\beta}\left(\int_{\mathbb{G}}\frac{(a+b|x|^{\alpha})^{\beta}}{|x|^{pm+p}}|f(x)|^{p}dx\right)^\frac{p-1}{p}
\left(\int_{\mathbb{G}}\frac{(a+b|x|^{\alpha})^{\beta}}{|x|^{pm}}|\mathcal{R}f(x)|^{p}dx\right)^\frac{1}{p},
$$
which gives \eqref{Lpweighted2}.

Now we show the sharpness of the constant. We need to check the equality
condition in above H\"older's inequality.
Let us consider the function
$$h(x)=|x|^{C}, $$
where $C\in\mathbb{R}, C\neq 0$ and $Q\neq pm+p-\alpha\beta$. Then by a direct calculation we obtain
\begin{equation}\label{Holder_eq2}
\left|\frac{1}{C}\right|^{p}\left(\frac{(a+b|x|^{\alpha})^
{\frac{\beta}{p}}|\mathcal{R}h(x)|}{|x|^{m}}\right)^{p}=\left(\frac{(a+b|x|^{\alpha})^
{\frac{\beta(p-1)}{p}}|h(x)|^{p-1}}
{|x|^{(m+1) (p-1)}}\right)^{\frac{p}{p-1}},
\end{equation}
which satisfies the equality condition in H\"older's inequality.
This gives the sharpness of the constant $\frac{Q-pm-p+\alpha\beta}{p}$
in \eqref{Lpweighted2}.
\end{proof}

\section{Higher order inequalities}
\label{Sec3}
In this section we present higher order $L^{p}$-Hardy type inequalities with super weights by iterating the obtained inequalities
\eqref{Lpweighted1} and \eqref{Lpweighted2}.
\begin{thm}\label{2}
Let $\mathbb{G}$ be a homogeneous group
of homogeneous dimension $Q$ and let $|\cdot|$ be a homogeneous quasi-norm on $\mathbb{G}$. Let $a,b>0$ and $1<p<\infty, Q\geq1, k\in \mathbb{N}$.
\begin{itemize}
\item[(i)] If $\alpha \beta>0$ and $pm\leq Q-p$, then for all $f\in C_{0}^{\infty}(\mathbb{G}\backslash\{0\})$, we have
\begin{equation}\label{Lpweighted3}
\left[\prod_{j=0}^{k-1}\left(\frac{Q-p}{p}-(m+j)\right)\right]
\left\|\frac{(a+b|x|^{\alpha})^{\frac{\beta}{p}}}{|x|^{m+k}}f\right\|_{L^{p}(\mathbb{G})}
\leq\left\|\frac{(a+b|x|^{\alpha})^{\frac{\beta}{p}}}{|x|^{m}}\mathcal{R}^{k}f\right\|_{L^{p}(\mathbb{G})}
.
\end{equation}

\item[(ii)] If $\alpha \beta<0$ and $pm-\alpha\beta\leq Q-p$, then for all $f\in C_{0}^{\infty}(\mathbb{G}\backslash\{0\})$, we have
\begin{equation}\label{Lpweighted4}
\left[\prod_{j=0}^{k-1}\left(\frac{Q-p+\alpha\beta}{p}-(m+j)\right)\right]
\left\|\frac{(a+b|x|^{\alpha})^{\frac{\beta}{p}}}{|x|^{m+k}}f\right\|_{L^{p}(\mathbb{G})}
\leq\left\|\frac{(a+b|x|^{\alpha})^{\frac{\beta}{p}}}{|x|^{m}}\mathcal{R}^{k}f\right\|_{L^{p}(\mathbb{G})}
.
\end{equation}
\end{itemize}
\end{thm}
In the case of $k=1$ \eqref{Lpweighted3} gives inequality \eqref{Lpweighted1} and \eqref{Lpweighted4} gives inequality \eqref{Lpweighted2}.
\begin{proof}[Proof of Theorem \ref{2}] We can iterate \eqref{Lpweighted1}, that is, we have
\begin{equation}\label{Lpiterate0}
\frac{Q-pm-p}{p}
\left\|\frac{(a+b|x|^{\alpha})^{\frac{\beta}{p}}}{|x|^{m+1}}f\right\|_{L^{p}(\mathbb{G})}
\leq\left\|\frac{(a+b|x|^{\alpha})^{\frac{\beta}{p}}}{|x|^{m}}\mathcal{R}f\right\|_{L^{p}(\mathbb{G})}.
\end{equation}
In \eqref{Lpiterate0} replacing $f$ by $\mathcal{R}f$ we obtain
\begin{equation}\label{Lpiterate1}
\frac{Q-pm-p}{p}
\left\|\frac{(a+b|x|^{\alpha})^{\frac{\beta}{p}}}{|x|^{m+1}}\mathcal{R}f\right\|_{L^{p}(\mathbb{G})}
\leq\left\|\frac{(a+b|x|^{\alpha})^{\frac{\beta}{p}}}{|x|^{m}}\mathcal{R}^{2}f\right\|_{L^{p}(\mathbb{G})}.
\end{equation}
On the other hand, replacing $m$ by $m+1$, \eqref{Lpiterate0} gives
\begin{equation}\label{Lpiterate2}
\frac{Q-p(m+1)-p}{p}
\left\|\frac{(a+b|x|^{\alpha})^{\frac{\beta}{p}}}{|x|^{m+2}}f\right\|_{L^{p}(\mathbb{G})}
\leq\left\|\frac{(a+b|x|^{\alpha})^{\frac{\beta}{p}}}{|x|^{m+1}}\mathcal{R}f\right\|_{L^{p}(\mathbb{G})}
.
\end{equation}
Combining this with \eqref{Lpiterate1} we obtain
\begin{multline*}\label{Lpiterate3}
\frac{Q-pm-p}{p}\cdot\frac{Q-p(m+1)-p}{p}
\left\|\frac{(a+b|x|^{\alpha})^{\frac{\beta}{p}}}{|x|^{m+2}}f\right\|_{L^{p}(\mathbb{G})} \\
\leq\left\|\frac{(a+b|x|^{\alpha})^{\frac{\beta}{p}}}{|x|^{m}}\mathcal{R}^{2}f\right\|_{L^{p}(\mathbb{G})}.
\end{multline*}
This iteration process gives
\begin{equation*}\label{Lpweighted5}
\prod_{j=0}^{k-1}\left(\frac{Q-p}{p}-(m+j)\right)
\left\|\frac{(a+b|x|^{\alpha})^{\frac{\beta}{p}}}{|x|^{m+k}}f\right\|_{L^{p}(\mathbb{G})}
\leq\left\|\frac{(a+b|x|^{\alpha})^{\frac{\beta}{p}}}{|x|^{m}}\mathcal{R}^{k}f\right\|_{L^{p}(\mathbb{G})}
.
\end{equation*}
Similarly, we have for $\alpha \beta<0$, $pm-\alpha\beta\leq Q-2$ and $f\in C_{0}^{\infty}(\mathbb{G}\backslash\{0\})$
\begin{equation*}\label{Lpweighted6}
\prod_{j=0}^{k-1}\left(\frac{Q-p+\alpha\beta}{p}-(m+j)\right)
\left\|\frac{(a+b|x|^{\alpha})^{\frac{\beta}{p}}}{|x|^{m+k}}f\right\|_{L^{p}(\mathbb{G})}
\leq\left\|\frac{(a+b|x|^{\alpha})^{\frac{\beta}{p}}}{|x|^{m}}\mathcal{R}^{k}f\right\|_{L^{p}(\mathbb{G})}
,
\end{equation*}
completing the proof.
\end{proof}


\begin{thebibliography}{HOHOLT08}

\bibitem[ACP05]{ACP05}
B.~Abdellaoui, E.~Colorado and I.~Peral.
\newblock Some improved Caffarelli-Kohn-Nirenberg inequalities.
\newblock {\em Calc. Var. Partial Differential Equations}, 23:327--345, 2005.

\bibitem[BJOS16]{BJOS16}
N.~Bez, C.~Jeavons, T.~Ozawa and M.~Sugimoto.
\newblock Stability of trace theorems on the sphere.
\newblock {\em arXiv:1611.00928}, 2016.

\bibitem[BL85]{Brez1}
H. Brezis and E. Lieb.
\newblock Inequalities with remainder terms.
\newblock {\em J. Funct. Anal.}, 62:73--86, 1985.

\bibitem[BM97]{Brez2}
H. Brezis and M. Marcus.
\newblock Hardy's inequalities revisited.
\newblock {\em Ann. Scuola
Norm. Sup. Pisa Cl. Sci.}, 25(4):217--237, 1997.


\bibitem[BV97]{BV97}
H.~Brezis and J.~L.~V\'{a}zquez.
\newblock Blow-up solutions of some nonlinear elliptic problems.
\newblock {\em Rev. Mat. Univ. Complut. Madrid.}, 10(2):443--469, 1997.

\bibitem[CF08]{CF08}
A.~Cianci and A.~Ferone.
\newblock Hardy inequalities with non-standard remainder terms.
\newblock {\em Ann. Inst. H. Poincar\'{e}. Anal. Nonlin\'{e}aire}, 25:889--906, 2008.

\bibitem[CKN84]{CKN84}
L.~A.~Caffarelli, R.~Kohn and L.~Nirenberg.
\newblock First order interpolation inequalities with weights.
\newblock {\em Composito Math.}, 53(3):259--275, 1984.

\bibitem[CFW13]{CFW13}
S.~Chen, R.~Frank and T.~Weth.
Remainder terms in the fractional Sobolev inequality.
{\em Indiana Univ. Math. J.}, 62:1381--1397, 2013.

\bibitem[FR16]{FR}
V.~Fischer and M.~Ruzhansky.
\newblock {\em Quantization on nilpotent Lie groups}, volume 314 of {\em
  Progress in Mathematics}.
\newblock Birkh\"auser, 2016.

\bibitem[FS82]{FS-Hardy}
G.~B. Folland and E.~M. Stein.
\newblock {\em Hardy spaces on homogeneous groups}, volume~28 of {\em
  Mathematical Notes}.
\newblock Princeton University Press, Princeton, N.J.; University of Tokyo
  Press, Tokyo, 1982.

\bibitem[FS08]{FrS08}
R.~L.~Frank and R.~Seiringer.
\newblock Non-linear ground state representations and sharp Hardy inequalities.
\newblock {\em J. Func. Anal.}, 255:3407--3430, 2008.

\bibitem[GM08]{GM08}
N.~Ghoussoub and A.~Moradifam.
\newblock On the best possible remaining term in the Hardy inequality.
\newblock {\em Proc. Natl. Acad. Sci. USA}, 105(37):13746--13751, 2008.

\bibitem[GM11]{GM11}
N.~Ghoussoub and A.~Moradifam.
\newblock Bessel pairs and optimal Hardy and Hardy-–Rellich
inequalities.
\newblock {\em Math. Ann.}, 349(1):1--57, 2011.

\bibitem[Han15]{Han15}
Y.~Han.
\newblock Weighted Caffarelli-Kohn-Nirenberg type inequality on the Heisenberg group.
\newblock {\em Indian J.Pure Appl. Math.}, 46(2):147--161, 2015.

\bibitem[HNZ11]{HNZh11}
Y.Z.~Han, P.C.~Niu, and T.~Zhang.
\newblock On first order interpolation inequalities with weights on the Heisenberg group.
\newblock {\em Acta Mathematica Sinica, English series}, 27(12):2493--2506, 2011.

\bibitem[HZ11]{HZh11}
Y.Z.~Han and Q.~Zhao.
\newblock A class of Caffarelli-Kohn-Nirenberg type inequalities for generalized Baouendi-Grushin vector fields.
\newblock {\em Acta Math. Sci. Ser. A Chin. Ed.}, 31(5):1181--1189, 2011 (in Chinese).

\bibitem[HZD11]{HZhD11}
Y.Z.~Han, S.T.~Zhang, and J.B.~Dou.
\newblock On first order interpolation inequalities with weights on the H-type group.
\newblock {\em Bull. Braz. Math. Soc.(N.S.)}, 42(2):185--202, 2011.

\bibitem[MOW15]{MOW15}
S.~Machihara, T.~Ozawa, and H.~Wadade.
\newblock Scaling invariant Hardy inequalities of multiple
logarithmic type on the whole space.
\newblock {\em J. Inequal. Appl.,} 281:1--13, 2015.

\bibitem[NDD12]{NDD12}
P.~N\'{a}poli, I.~Drelichman and R.~Dur\'{a}n.
Improved Caffarelli-Kohn-Nirenberg and trace inequalities for radial functions.
{\em Commun. Pure Appl. Anal.}, 11:1629–1642, 2012.

\bibitem[ORS16]{ORS16}
T.~Ozawa, M.~Ruzhansky and D.~Suragan.
\newblock $L^{p}$-Caffarelli-Kohn-Nirenberg type inequalities on homogeneous groups.
\newblock {\em arXiv:1605.02520}, 2016.

\bibitem[RS16a]{Ruzhansky-Suragan:L2-CKN}
M.~Ruzhansky and D.~Suragan.
\newblock Anisotropic $L^{2}$-weighted Hardy and $L^{2}$-Caffarelli-Kohn-Nirenberg inequalities.
\newblock {\em Commun. Contemp. Math.}, 2016. http://dx.doi.org/10.1142/S0219199717500146

\bibitem[RS16b]{Ruzhansky-Suragan:critical}
M.~Ruzhansky and D.~Suragan.
\newblock Critical {H}ardy inequalities.
\newblock {\em arXiv: 1602.04809}, 2016.

\bibitem[RS16c]{Ruzhansky-Suragan:squares}
M.~Ruzhansky and D.~Suragan.
\newblock Local {H}ardy and Rellich inequalities for sums of squares of vector fields.
\newblock {\em Adv. Diff. Equations}, to appear, 2016.
arXiv:1605.06389

\bibitem[RS17]{Ruzhansky-Suragan:JDE}
M.~Ruzhansky and D.~Suragan.
On horizontal Hardy, Rellich, Caffarelli-Kohn-Nirenberg and $p$-sub-Laplacian inequalities on stratified groups.
{\em J. Differential Equations}, 262:1799--1821, 2017.

\bibitem[RS17a]{Ruzhansky-Suragan:Layers}
M.~Ruzhansky and D.~Suragan.
\newblock Layer potentials, {K}ac's problem, and refined {H}ardy inequality on
homogeneous {C}arnot groups.
\newblock {\em Adv. Math.}, 308:483--528, 2017.

\bibitem[RSY16]{RSY16}
M.~Ruzhansky, D.~Suragan and N.~Yessirkegenov.
\newblock Sobolev inequalities, Euler-Hilbert-Sobolev and Sobolev-Lorentz-Zygmund spaces on homogeneous groups.
\newblock {\em 	 arXiv:1610.03379v2}, 2016.

\bibitem[ST15a]{ST15a}
M.~Sano and F.~Takahashi.
\newblock Scale invariance structures of the critical and the subcritical Hardy inequalities and their improvements, preprint, 2015.
http://www.sci.osaka-cu.ac.jp/math/OCAMI/preprint/2015/$15_{-}$05.pdf

\bibitem[ST15b]{ST15b}
M.~Sano and F.~Takahashi.
\newblock Some improvements for a class of the Caffarelli-Kohn-Nirenberg inequalities, preprint, 2015.
 http://www.sci.osaka-cu.ac.jp/math/OCAMI/preprint/2015/$15_{-}$13.pdf.

\bibitem[WW03]{WW03}
Z-Q.~Wang and M.~Willem.
\newblock Caffarelli-Kohn-Nirenberg inequalities with remainder terms.
\newblock {\em J. Func. Anal.}, 203:550--568, 2003.

\bibitem[Yac17]{Yacoub17}
Ch.~Yacoub.
\newblock Caffarelli-Kohn-Nirenberg inequalities on Lie groups of polynomial growth.
\newblock {\em arXiv:1702.04969}, 2017.

\bibitem[ZHD14]{ZhHD14}
Sh.~Zhang, Y.~Han and J.~Dou.
\newblock A class of Caffarelli-Kohn-Nirenberg type inequalities on the H-type group.
\newblock {\em Rend.Sem.Mat.Univ.Padova}, 132:249--266, 2014.

\bibitem[ZHD15]{ZhHD15}
Sh.~Zhang, Y.~Han and J.~Dou.
\newblock Weighted Hardy-Sobolev type inequality for generalized Baouendi-Grushin vector fields and its application.
\newblock {\em Advances in Mathematics (China)}, 44(3):411--420, 2015.


\end{thebibliography}
\end{document}